\documentclass[11pt]{amsart}
\usepackage{latexsym, amsfonts,mathrsfs, amsmath, amssymb, amscd, epsfig}
\usepackage{mathrsfs}

\usepackage{amsthm}
\usepackage[all,cmtip,line]{xy}
\usepackage{array}
\usepackage{psfrag}
\usepackage{overpic}
\usepackage{bm}
\usepackage{setspace}

\newtheorem{theorem}{Theorem}%[section]
\newtheorem{lemma}{Lemma}[section]
\newtheorem{proposition}[lemma]{Proposition}%[section]
\newtheorem{corollary}[lemma]{Corollary}%[section]
\newtheorem{definition}[lemma]{Definition}%[section]
\newtheorem{remark}[lemma]{Remark}%[section]
\newtheorem{problem}{{\bf{Problem}}}%[section]

\theoremstyle{definition}

\def\beq#1\eeq{\begin{equation}#1\end{equation}}
\def\balign #1 #2 \ealign{\begin{aligned} #1 #2  \end{aligned} }

\def\Div{{\rm div}}
\def\vol{{\rm vol}}
\def\sgn{{\rm sgn}}

\def\bu{\mathbf{u}}
\def\bI{\mathbf{I}}
\def\mE{\mathcal{E}}

\newcommand \alp{\alpha}
\newcommand \eps{\varepsilon}
\newcommand \vphi{\varphi}

\newcommand \Gam{\Gamma}
\newcommand \gam{\gamma}

\newcommand \tx{\text}
\newcommand \R{\mathbb{R}}
\newcommand \til{\tilde}

\newcommand \der{\partial}
\newcommand \mcl{\mathcal}
\newcommand \ol{\overline}
\newcommand \Om{\Omega}
\newcommand \N{\mcl{N}}
\newcommand \q{{\bf q}}

\newcommand \mH{\mcl{H}}

\newcommand \corners{\Gam_0\cup\Gam_L}
\newcommand \tpsi{\til{\psi}}
\newcommand \tPsi{\til{\Psi}}

\newcommand \pex{p_{ex}}

\newcommand \rx{{\rm x}}
\newcommand \ry{{\rm y}}

\def \msB {\mathscr{B}}
\def\mff{\mathfrak{f}}
\def\mfh{\mathfrak{h}}

\numberwithin{equation}{section}

\begin{document}
%\tableofcontents

\title[Subsonic flow for multidimensional Euler-Poisson system]
%[Application of Gradient estimates for elliptic system to Euler-Poisson system]
{Subsonic flow for multidimensional Euler-Poisson system}
%{Gradient estimates for second order elliptic system and its application to Euler-Poisson system}

\author{Myoungjean Bae}
\address{M. Bae, Department of Mathematics\\
         POSTECH\\
         San 31, Hyojadong, Namgu, Pohang, Gyungbuk, Korea}
\email{mjbae@postech.ac.kr}
\author{Ben Duan}
\address{B. Duan, Department of Mathematics\\
         POSTECH\\
         San 31, Hyojadong, Namgu, Pohang, Gyungbuk, Korea}
\email{bduan@postech.ac.kr}

\author{Chunjing Xie}
\address{C. Xie, Department of mathematics, Institute of Natural Sciences, Ministry of Education Key Laboratory of Scientific and Engineering Computing, Shanghai Jiao Tong University\\
800 Dongchuan Road, Shanghai, China
}
\email{cjxie@sjtu.edu.cn}

\begin{abstract}
We establish unique existence and stability of subsonic potential flow for steady Euler-Poisson system in a multidimensional nozzle of a finite length when prescribing the electric potential difference on non-insulated boundary from a fixed point at the exit, and prescribing the pressure at the exit of the nozzle. The Euler-Poisson system for subsonic potential flow can be reduced to a nonlinear elliptic system of second order. In this paper, we develop a technique to achieve a priori $C^{1,\alp}$ estimates of solutions to a quasi-linear second order elliptic system with mixed boundary conditions in a multidimensional domain with Lipschitz continuous boundary. Particularly, we discovered a special structure of the Euler-Poisson system which enables us to obtain $C^{1,\alp}$ estimates of velocity potential and electric potential functions, and this leads us to establish structural stability of subsonic flows for the Euler-Poisson system under perturbations of various data.

\end{abstract}

\date{\today}

\maketitle
%\tableofcontents

\section{Introduction and Main Results}
\subsection{Preliminary}

The following nonlinear system, called the \emph{Euler-Poisson system},
\begin{equation}\label{UnsteadyEP}
\left\{
\begin{aligned}
& \rho_t+ \Div (\rho \bu)=0, \\
& (\rho \bu)_t+\Div (\rho \bu \otimes \bu +p{\mathbf{I}}_n)=\rho \nabla \Phi, \\
& (\rho \mE)_t +\Div(\rho\mE \bu +p\bu)=\rho \bu\cdot \nabla\Phi,\\
& \Delta \Phi=\rho-b(x),%
\end{aligned}%
\right.
\end{equation}
models various physical phenomena including the propagation of electrons in
submicron semiconductor devices and plasmas (cf.\cite{MarkRSbook}), and the biological transport of ions for channel proteins (cf.\cite{Shu}). In the hydrodynamical model of semiconductor devices or plasmas, $\bu,
\rho$, $p$, and $\mE$ represent the macroscopic particle velocity, electron density,
pressure, and the total energy, respectively.  The electric potential $\Phi$ is  generated by the
Coulomb force of particles. $\mathbf{I}_n$ is the $n\times n$ identity matrix and $b({\rx})>0$ stands for the density of fixed,
positively charged background ions. In the biological model describing transport of ions between the extracellular side and the cytoplasmic side of
the membranes(cf. \cite{Shu}), $\rho$, $\rho \bu$ and $\Phi$ are the ion
concentration, the ion translational mass, and the electric potential,
respectively. The system \eqref{UnsteadyEP} is closed with the aid of definition of total energy and the equation of state
\begin{equation}
\mE=\frac{|\bu|^2}{2}+e,\quad p=p(\rho, e),
\end{equation}
where $e$ is the internal energy. By the second law of thermodynamics, the entropy $S$ is a constant along each particle trajectory in smooth steady flow of \eqref{UnsteadyEP}. So the entropy $S$ is globally a constant if $S$ is a constant at the entrance of flow. In this paper, we consider the case when the entropy $S$ is globally a constant, and such a case is called \emph{isentropic}. For isentropic flow, we can write the pressure $p$ and the internal energy $e$ as
\begin{equation}
\label{2-a8}
p=p(\rho),  \quad e=\int^{\rho}_{k_0} \frac{p'(\varrho)}{\varrho} d\varrho-\frac{p(\rho)}{\rho}
\end{equation}
for a constant $k_0>0$. We assume that $p\in C^{0}([0,+\infty))\cap C^{3}((0,\infty))$ satisfies:
\begin{equation}
\label{2-a5}
p(0)=0,\quad  p'(\rho)>0, \quad p''(\rho)\ge 0\quad\tx{for all}\quad \rho>0\quad\tx{and}\quad p(+\infty)=+\infty.
\end{equation}
For a constant $\gam\ge 1$, $p(\rho)=\rho^{\gam}$ is a typical example of $p$ satisfying \eqref{2-a5}, which corresponds to polytropic gas in gas dynamics.

The steady Euler-Poisson system of isentropic flow is
\begin{equation}\label{SteadyEP}
\begin{cases}
& \Div(\rho \bu)=0, \\
& \Div (\rho \bu\otimes \bu +p{\bI}_n)=\rho \nabla\Phi, \\
%&\Div(\rho \bu \mE+\bu p)=\rho \bu \cdot \nabla\Phi,\\
& \Delta \Phi=\rho -b%
\end{cases}
\end{equation}%
for a given function $b>0$.
There are several issues to make the system \eqref{SteadyEP} complicated. The first is that \eqref{SteadyEP} is a mixed type system, and its type depends on \emph{the Mach number} $M$ which is given by $M=\frac{|\bu|}{c(\rho)}$ for $c(\rho)=\sqrt{p'(\rho)}$. Here, $c$ is called the \emph{local sound speed}. If $M<1$, then \eqref{SteadyEP} can be decomposed into a nonlinear elliptic system and homogeneous transport equations, and the flow is said \emph{subsonic}. If $M>1$, on the other hand, \eqref{SteadyEP} can be decomposed into a nonlinear hyperbolic-elliptic coupled system and homogeneous transport equations at best, and the flow is said \emph{supersonic}. The second issue is that the last equation in \eqref{SteadyEP}, which is a \emph{Poisson equation}, has a nonlocal effect to the other equations in \eqref{SteadyEP}, and it makes the fluid variables $\rho, {\bf u}$ and electric potential $\Phi$ interact in a highly nonlinear way. Also, physical boundary conditions such as fixed exit pressure give nonlinear boundary conditions for the system \eqref{SteadyEP}.

It is our goal to prove unique existence and stability of subsonic flows for steady Euler-Poisson system in a multidimensional nozzle under perturbations of exit pressure, electric potential difference on non-insulated boundary and under perturbations of the nozzle itself. Our motivation is originated from the study on structural stability of transonic shocks occurring in flow governed by steady Euler-Poisson system. There have been a few studies on transonic shocks of Euler-Poisson system(cf.\cite{MarkPhase, Gamba1d, LRXX, LuoXin, RosiniPhase}). In \cite{LuoXin}, the authors considered one dimensional solutions of \eqref{SteadyEP} with a constant background charge $b(x)=b_0>0$, and proved the unique existence of transonic shock solutions provided that the entrance and exit data are appropriately given. So it is natural to ask whether these one dimensional transonic shocks are dynamically and structurally stable. The dynamical stability of one dimensional transonic shock solution is achieved in \cite{LRXX}. In order to study structural stability of transonic shocks, however, it is inevitable to consider small perturbations of one dimensional transonic shocks in multidimensional domain, but there are very few known results about multidimensional solutions of Euler-Poisson system(cf.\cite{Gamba1d, GambaMorawetz}). Comparing with extensive studies and recent significant progress on transonic shock solutions of the Euler system(see \cite{Ch-F3, Schen-Yu1, XinYinCPAM} and references therein), stability problems for multidimensional transonic flows of the Euler-Poisson system are essentially open. The main difference of the Euler-Poisson system from the Euler system is that the Poisson equation for electric potential is coupled with the other equations in the Euler-Poisson system. And, this makes it hard to analyze even one dimensional solution of the Euler-Poisson system. In fact, one dimensional flow of the Euler-Poisson system behaves very differently from the one of the Euler system(cf.\cite{LuoXin}). And it is even harder to study multidimensional transonic flow of the Euler-Poisson system due to nonlinear interaction between the electric potential $\Phi$ and all the other fluid variables.

As the first step to investigate stability of multidimensional transonic flow of the Euler-Poisson system, we establish the unique existence and stability of subsonic flows of steady Euler-Poisson system under perturbations of the exit pressure and electric potential difference on non-insulated boundary. There have been a few results about existence of subsonic solution of hydrodynamic equations, which are the Euler-Poisson system with relaxation, under smallness assumptions on the flow velocity for both unsteady and steady cases(see \cite{DeMark1d, DeMark3d, Guo, MarkZAMP}). In this paper, we prove existence of multidimensional subsonic solutions of steady Euler-Poisson system without smallness assumption on the flow velocity.

%\newpage

 Fix an open, bounded and connected set $\Lambda\subset \R^{n-1}(n\ge 2)$ with a smooth boundary $\der\Lambda$, and define a nozzle $\N$ by
\begin{equation}
\label{domain}
\N:=\{\rx=(\rx',x_n)\in\R^n: \rx'\in\Lambda,\;\; x_n\in(0, L)\}\subset \R^n.
\end{equation}
The nozzle boundary $\partial\N$ consists of the entrance $\Gamma_0=\Lambda\times\{0\}$, the exit $\Gamma_L=\Lambda\times\{L\}$, and the insulated boundary $\Gamma_w=\partial \Lambda\times (0, L)$.
In order to study the system \eqref{SteadyEP} in a multidimensional domains $\N$ with arbitrary cross-section $\Lambda$, we consider irrotational flow where the velocity $\bu$ of the flow is represented by
 \begin{equation}
 \label{2-b3}
 \bu=\nabla\vphi
 \end{equation}
for a scalar function $\vphi$ which is called a \emph{velocity potential} function. By \eqref{2-a8} and \eqref{2-b3}, the second equation in \eqref{SteadyEP} can be rewritten as
\begin{equation}
\label{2-b2n}
\rho \nabla (\msB-\Phi)=0
\end{equation}
for
\begin{equation}\label{2-a4}
\msB=\frac{1}{2}|\nabla\vphi|^2 +\int_{k_0}^{\rho}\frac{p'(\varrho)}{\varrho} d\varrho.
\end{equation}
For $\rho>0$, \eqref{2-b2n} implies
\begin{equation}
\label{2-b3n}
\msB-\Phi\equiv K_0
\end{equation}
for a constant $K_0$. Without loss of generality, we choose $K_0=0$. Set
\begin{equation}
\label{rho-eqn}
h(\rho):=\int_{k_0}^{\rho}\frac{p'(\varrho)}{\varrho}d\varrho.
\end{equation}
Then, the equation \eqref{2-b3n} with $K_0=0$ implies that $h(\rho)=\Phi-\frac 12|\nabla\vphi|^2.$
From the condition \eqref{2-a5}, we have $h'(\rho)>0$ for all $\rho>0$ and $\underset{\rho\to \infty}{\lim}\;h(\rho)=\infty$. So the function $h$ is invertible wherever $\Phi-\frac 12|\nabla\vphi|^2>\underset{\rho\to 0+}{\liminf}\;h(\rho)$, in which case one can rewrite \eqref{rho-eqn} as
%\begin{equation}
%\label{2-b4n}
$\rho=h^{-1}(\Phi-\frac 12|{\nabla\vphi}|^2)$.
%\end{equation}
We use this expression to reduce \eqref{SteadyEP} to a nonlinear system of second order equations for $\vphi$ and $\Phi$:
\begin{align}
\label{2-b2}
&\Div\bigl(\rho(\Phi,|\nabla\vphi|^2)\nabla\vphi\bigr)=0,\\
\label{2-b3d}
&\Delta \Phi= \rho(\Phi,|\nabla\vphi|^2)-b
\end{align}
with $\rho>0$ given by
\begin{equation}
\label{2-b4}
\rho(\Phi,|\nabla\vphi|^2)=h^{-1}(\Phi-\frac 12|\nabla\vphi|^2)
\end{equation}
provided that $h^{-1}$ is well defined. If we regard \eqref{2-b2} as an equation for $\vphi$, then it is mixed type. More precisely, \eqref{2-b2} is elliptic if and only if
\begin{equation}
\label{2-b5}
|\nabla\vphi|^2<p'(\rho)\tx{\emph{(subsonic)}}
\end{equation}
and is {\emph{hyperbolic}} if and only if
\begin{equation}
\label{2-b5-d}
|\nabla\vphi|^2>p'(\rho)\tx{\emph{(supersonic)}}.
\end{equation}
The system of \eqref{2-b2} and \eqref{2-b3d} becomes a quasilinear elliptic system if \eqref{2-b5} holds, and a hyperbolic-elliptic coupled system if \eqref{2-b5-d} holds.

Our interest is on stability of subsonic solution under perturbations of exit pressure and electric potential difference on non-insulated boundary from a fixed point. So the boundary conditions are formulated as follows. First, for a given function $\pex$ on $\Gam_L$, set
\begin{equation}
\label{2-b8}
p\bigl(\rho(\Phi,|\nabla\vphi|^2)\bigr)=p_{ex} \quad\tx{on}\quad \Gamma_L.
\end{equation}
 On the wall boundary, slip/insulated boundary conditions for $\vphi$ and $\Phi$ are prescribed as follows:
\begin{equation}
\label{2-b9-2}
\der_{{\bf n}_w}\vphi=\der_{{\bf n}_w}\Phi=0\;\;\tx{on}\;\;\Gam_w
\end{equation}
where  ${\bf n}_w$ is the unit inward normal vector  on $\Gam_w$. We fix a point $\rx_0$ on $\Gam_L$, and prescribe the electric potential difference between two points $\rx\in\Gam_0\cup\Gam_L$ and $\rx_0$ as follows:
\begin{equation}\label{pdiff}
\Phi(\rx)-\Phi(\rx_0)=
\begin{cases}\bar\Phi_{en}(\rx)&\text{for } \rx\in \Gamma_0\\
\bar\Phi_{ex}(\rx)&\text{for } \rx\in\Gamma_L
\end{cases}.
\end{equation}
In \eqref{pdiff}, the value of $\Phi(\rx_0)$ is uniquely determined by \eqref{2-b3n} and one point boundary condition for the Bernoulli's function:
 \begin{equation}\label{bcpressure}
\msB(\rx_0)=\msB_0.
\end{equation}
\begin{remark}
Differently from the Euler system, the Bernoulli's function $\msB=\frac 12|\nabla\vphi|^2+\int_{k_0}^{\rho}\frac{p'(\varrho)}{\varrho}d\varrho$ is not a constant along each streamline due to the equation \eqref{2-b2n}. So we call $\msB$ the Bernoulli's function rather than the Bernoulli's invariant. This is one of differences of the Euler-Poisson system from the Euler system.
\end{remark}
Finally, homogeneous Dirichlet boundary condition for $\vphi$ is imposed at the entrance:
\begin{equation}\label{bcvphi}
\vphi=0\quad \text{ on } \quad\Gamma_0.
\end{equation}
The boundary data $p_{ex}, \bar{\Phi}_{en}, \bar{\Phi}_{ex}$ and $\msB_0$ will be specified later.

\subsection{Main theorem}
 %The objective of this work is to prove unique existence and stability of subsonic solution of \eqref{2-b2} and \eqref{2-b3d} in the multidimensional nozzle $\N$ with the boundary conditions \eqref{2-b8}--\eqref{bcvphi}.
%\begin{problem}
%\label{problem1}%(an equivalent version of Problem \ref{problem0})
%Given $p_{ex}, \bar{\Phi}_{en}, \bar{\Phi}_{ex}$ and $\msB_0$, find a solution $(\vphi,\Phi)$ to the system of \eqref{2-b2} and \eqref{2-b3d} in $\N$ with the boundary conditions \eqref{2-b8}--\eqref{bcvphi} so that the solution $(\vphi,\Phi)$ satisfies the inequality
%\begin{equation}
%\label{subsonic-ineq}
%|\nabla\vphi|^2<p'(\rho(\Phi,|\nabla\vphi|^2))\quad\tx{in}\quad \ol{\N}.
%\end{equation}
%\end{problem}
If we fix $b$ as a constant $b_0>0$ in the equation \eqref{2-b3d} then the equations \eqref{2-b2} and \eqref{2-b3d} become invariant under translation. So if the boundary data $\bar\Phi_{en}$, $\bar\Phi_{ex}$ and $p_{ex}$ are all constants, then one may look for a solution $(\vphi,\Phi)$ as functions of $x_n$ only for $x_n\in (0,L)$. We note that $\bar\Phi_{ex}=0$ if $\bar\Phi_{ex}$ is a constant.
\begin{proposition}[\emph{One dimensional subsonic flow}]
\label{lemma-2-1}
Fix constants $b_0>0$ and $L>0$. Then there exists a nonempty set $\mathfrak{P}_0$ of parameters in $\R^2\times \R^+$ so that for any $(\bar\Phi_{en,0}, \msB_{0,0}, p_{ex,0})\in \mathfrak{P}_0$, if $(\bar\Phi_{en},\msB_0, p_{ex})=(\bar{\Phi}_{en,0},\msB_{0,0},p_{ex,0})$ then the system of \eqref{2-b2} and \eqref{2-b3d} in $\N$ with the boundary conditions \eqref{2-b8}--\eqref{bcvphi} has a unique $C^2$ one-dimensional solution $(\vphi,\Phi)$ in $\ol{\N}$ satisfying the inequalities $\rho(\Phi, |\nabla\vphi|^2)>0$ and $|\nabla\vphi|^2<p'(\rho(\Phi,|\nabla\vphi|^2))$ in $\ol{\N}$.
\end{proposition}
The proof of Proposition \ref{lemma-2-1} is in Section \ref{appendix-pf}.

\begin{remark}
We can find one-dimensional solutions for a nonconstant function $b$ provided that the boundary data $(\bar\Phi_{en,0}, \msB_{0,0}, p_{ex,0})$ are properly chosen depending on $b$. Details can be found in Appendix \ref{section-nonconstant-b}.
\end{remark}

\begin{definition}[Background solution]
\label{def-background-sol}
Fix constants $b_0>0$ and $L>0$, and let the parameter set $\mathfrak{P}_0$ be as in Proposition \ref{lemma-2-1}. For given constants $(\bar\Phi_{en,0}, \msB_{0,0}, p_{ex,0})\in \mathfrak{P}_0$, let $(\vphi,\Phi)(x_n)$ be the corresponding one-dimensional solution to the system of \eqref{2-b2} and \eqref{2-b3d} in $\N$ with the boundary conditions \eqref{2-b8}--\eqref{bcvphi} for the boundary data $(\bar\Phi_{en,0}, \msB_{0,0}, p_{ex,0})$. For such solution $(\vphi,\Phi)(x_n)$, we define two functions $(\vphi_0,\Phi_0)(x_1,\cdots,x_n)$ in $\N$ by
\begin{equation}
\label{2-b6}
\begin{split}
&\vphi_0(x_1,\cdots, x_n):=\vphi(x_n),\\
&\Phi_0(x_1,\cdots, x_n):=\Phi(x_n).
\end{split}
\end{equation}
We call $(\vphi_0,\Phi_0)$ the background solution corresponding to $(\bar\Phi_{en,0}, \msB_{0,0}, p_{ex,0})\in \mathfrak{P}_0$.
%\begin{remark}
%\label{remark-smooth-bs}
$\vphi_0$ and $\Phi_0$ are smooth in $\ol{\N}$ and satisfy
\begin{equation}
\label{background-sol-prop}
\rho(\Phi_0,|\nabla\vphi_0|^2)\ge \nu_0,\quad p'(\rho(\Phi_0,|\nabla\vphi_0|^2))-|\nabla\vphi_0|^2\ge \nu_0\quad\tx{in}\quad\ol{\N}
\end{equation}
for a constant $\nu_0>0$ depending on $b_0$, $L$, $\bar\Phi_{ex}$, $\msB_{0,0}$ and $p_{ex,0}$.
%\end{remark}
\end{definition}
Our goal is to prove stability of a background solution under small perturbations of the background charge $b$ and the boundary data. In Appendix \ref{section-perturbation}, we also consider stability of a background solution under small perturbations of the nozzle $\N$.
\begin{problem}
\label{problem2}
Fix $b_0>0$ and $L>0$, and let the parameter set $\mathfrak{P}_0$ be as in Proposition \ref{lemma-2-1}. Given $(\bar\Phi_{en,0}, \msB_{0,0}, p_{ex,0})\in\mathfrak{P}_0$, let $(\vphi_0,\Phi_0)$ be the corresponding background solution.
Also, let $(\bar\Phi_{en},\bar\Phi_{ex},p_{ex})$ be small perturbations of $(\bar\Phi_{en,0},0,\bar p_{ex,0})$ and let a constant $\msB_0$ be close to $\msB_{0,0}$. And, let $\rho(\Phi,|\nabla\vphi|^2)>0$ be defined by \eqref{2-b4}. Find a solution $(\vphi,\Phi)$ to the nonlinear boundary value problem:
\begin{align}
\label{nbvp-1}
&\begin{cases}
\Div (\rho(\Phi,|\nabla\vphi|^2)\nabla\vphi)=0\\
\Delta\Phi=\rho(\Phi,|\nabla\vphi|^2)-b
\end{cases}\quad\tx{in}\quad\N\\
\label{nbvp-2}
&\vphi=0\;\;\tx{on}\;\;\Gam_0, \quad p(\rho(\Phi,|\nabla\vphi|^2))=p_{ex}\;\;\tx{on}\;\;\Gam_L\\
\label{nbvp-3}
&\Phi=\begin{cases}
\msB_0+\bar{\Phi}_{en}&\tx{on}\;\;\Gam_0\\
\msB_0+\bar{\Phi}_{ex}&\tx{on}\;\;\Gam_L
\end{cases}\\
\label{nbvp-4}
&\der_{{\bf n}_w}\vphi=\der_{{\bf n}_w}\Phi=0\;\;\tx{on}\;\;\Gam_w.
\end{align}
\end{problem}
In order to state our main results on Problem \ref{problem2}, weighted H\"{o}lder norms are introduced first.

For a bounded connected open set $\Om\subset\R^n$, let $\Gam$ be a closed portion of $\der\Om$. For $\rx,\ry\in\Om$, set
\begin{equation*}
\delta_{\rx}:=\rm{dist}(\rx,\Gam)\quad \text{and}\quad  \delta_{\rx,\ry}:=\min(\delta_{\rx},\delta_{\ry}).
\end{equation*}
For $k\in\R$, $\alp\in(0,1)$ and $m\in \mathbb{Z}^+$, define the standard H\"{o}lder norms by
\begin{align*}
&\|u\|_{m,0,\Om}:=\sum_{0\le|\beta|\le m}\sup_{\rx\in \Om}|D^{\beta}u(\rx)|,\quad
[u]_{m,\alp,\Om}:=\sum_{|\beta|=m}\sup_{\rx, \ry\in\Om,\rx\neq  \ry}\frac{|D^{\beta}u(\rx)-D^{\beta}u(\ry)|}{|\rx-\ry|^{\alp}},
\end{align*}
and the weighted H\"{o}lder norms by
\begin{align*}
&\|u\|_{m,0,\Om}^{(k,\Gam)}:=\sum_{0\le|\beta|\le m}\sup_{\rx\in \Om}\delta_{\rx}^{\max(|\beta|+k,0)}|D^{\beta}u(\rx)|,\\
&[u]_{m,\alp,\Om}^{(k,\Gam)}:=\sum_{|\beta|=m}\sup_{\rx,\ry\in\Om, \rx\neq \ry}\delta_{\rx,\ry}^{\max(m+\alp+k,0)}\frac{|D^{\beta}u(\rx)-D^{\beta}u(\ry)|}{|\rx-\ry|^{\alp}},\\
&\|u\|_{m,\alp,\Om}:=\|u\|_{m,0,\Om}+[u]_{m,\alp,\Om},\quad
\|u\|_{m,\alp,\Om}^{(k,\Gam)}:=\|u\|_{m,0,\Om}^{(k,\Gam)}+[u]_{m,\alp,\Om}^{(k,\Gam)},
\end{align*}
where $D^{\beta}$ denotes $\der_{x_1}^{\beta_1}\cdots\der_{x_n}^{\beta_n}$ for a multi-index $\beta=(\beta_1,\cdots,\beta_n)$ with $\beta_j\in\mathbb{Z}_+$ and $|\beta|=\sum_{j=1}^n\beta_j$. $C^{m,\alp}_{(k,\Gam)}(\Om)$ denotes the completion of the set of all smooth functions whose $\|\cdot\|_{m,\alp,\Om}^{(k,\Gam)}$ norms are finite.

The main result of this work is the following.
\begin{theorem}
\label{theorem1}Let $\N$ be as in \eqref{domain}.
Fix $b_0>0$ and $L>0$, and let the parameter set $\mathfrak{P}_0$ be as in Proposition \ref{lemma-2-1}. Given $(\bar\Phi_{en,0}, \msB_{0,0}, p_{ex,0})\in\mathfrak{P}_0$, let $(\vphi_0,\Phi_0)$ be the corresponding background solution. Assume that $b,(\bar\Phi_{en},\bar\Phi_{ex},p_{ex})$ and $\msB_0$ are given as small perturbations of $b_0$, $(\bar\Phi_{en,0},0,\bar p_{ex,0})$ and $\msB_{0,0}$, respectively, in the following sense:
\begin{equation}
\label{data}
\begin{split}
&\|b-b_0\|_{\alp,\N}\le\sigma,\\
&\|\bar\Phi_{en}-\bar\Phi_{en,0}\|_{2,\alp,\Gam_0}
+\|\bar\Phi_{ex}\|_{2,\alp,\Gam_L}+\|\pex- p_{ex,0}\|_{\alp,\Gam_L} \le \sigma,\\
&|\msB_0-\msB_{0,0}|\le \sigma
\end{split}
\end{equation}
for a small constant $\sigma>0$ to be specified below. Also, suppose that $\bar{\Phi}_{en}$ and $\bar\Phi_{ex}$ satisfy the compatibility conditions
\begin{equation}
\label{compatibility}
\der_{{\bf n}_w}\bar\Phi_{en}=0\;\;\tx{on}\;\;\ol{\Gam}_0\cap\ol{\Gam}_w,\quad
\der_{{\bf n}_w}\bar{\Phi}_{ex}=0\;\;\tx{on}\;\;\ol{\Gam}_L\cap\ol{\Gam}_w.
\end{equation}
Then, for any given $\alp\in(0,1)$, there exists a constant $\bar\sigma>0$ depending on $b_0, L, \bar\Phi_{en,0}, \msB_{0,0}, p_{ex,0}$ and $\alp$ such that wherever $\sigma\in(0,
\bar\sigma]$, if the boundary data and $b$ satisfy \eqref{data} and \eqref{compatibility},
then the nonlinear boundary value problem \eqref{nbvp-1}--\eqref{nbvp-4} has a unique solution $(\vphi,\Phi)\in [C^{1,\alp}(\ol\N)\cap C^{2,\alp}(\N)]^2$ satisfying the following properties:
\begin{itemize}
\item[(a)] The equations in \eqref{nbvp-1} form a uniformly elliptic system in $\N$. Equivalently, the solution $(\vphi,\Phi)$ satisfies the inequality
    \begin{equation*}
    p'(\rho(\Phi,|\nabla\vphi|^2))-|\nabla\vphi|^2\ge \bar\nu>0\quad\tx{in}\quad\ol{\N}
    \end{equation*}
    for a positive constant $\bar{\nu}$, i.e., the flow governed by $(\vphi,\Phi)$ is subsonic;

\item[(b)] $(\vphi,\Phi)$ satisfy the estimate
   \begin{equation}
\label{2-c1}
\|\vphi-\vphi_0\|_{2,\alp,\N}^{(-1-\alp,\corners)}+
\|\Phi-\Phi_0\|_{2,\alp,\N}^{(-1-\alp, \corners)}\le C\sigma,
\end{equation}
for
$\sigma$ in \eqref{data}.
The constants $\bar\nu$ and $C$ depend only on  $b_0, L, \bar\Phi_{en,0}, \msB_{0,0}, p_{ex,0}, n,\Lambda$ and $\alp$.
\end{itemize}
\end{theorem}

For simplicity, we hereafter say that a constant $C$ depends on the data if $C$ depends only on  $b_0, L, \bar\Phi_{en,0}, \msB_{0,0}, \bar p_{ex,0}, n,\Lambda$ and $\alp$.

We point out that the boundary conditions \eqref{2-b8}--\eqref{bcvphi} are physically measurable. In one-dimensional solutions, they correspond to prescribing the pressure(or equivalently prescribing the density) at both ends of the nozzle. See Remark \ref{remark-1a} for details.

In order for a solution $(\vphi,\Phi)$ of \eqref{2-b2} and \eqref{2-b3d} to satisfy \eqref{2-b5}, it is essential to establish $L^{\infty}$ estimates of $(D\vphi,\Phi)$.
It is the new feature of this work that we prove $C^{1,\alp}$ regularity of solutions to a class of second order elliptic systems with mixed boundary conditions on a Lipschitz continuous boundary, and use the result to find a solution $(\vphi,\Phi)$ of \eqref{2-b2} and \eqref{2-b3d} so that $(\vphi,\Phi)$ satisfy \eqref{2-b5} in the nozzle $\N$. Furthermore, the $C^{1,\alp}$ estimates that we achieve applies up to the boundary $\der\N$. Differently from elliptic equations, $C^{1,\alp}$ regularity of second order elliptic system is not generally known. In the Euler-Poisson system, however, we discovered a special structure(see property (c) of Lemma \ref{lemma-3-1}), and use it and the divergence structure of the Euler-Poisson system to get a priori $H^1$ estimate of weak solutions to second order elliptic system which yields a priori $C^{1,\alp}$ estimates of weak solutions to second order elliptic systems with mixed boundary conditions in the Lipschitz domain $\N$. Thanks to this property, we can prove unique existence and stability of multidimensional subsonic solutions of Euler-Poisson system without smallness assumptions on the current flux or flow velocity. This distinguishes the results in this paper from previous works in  \cite{DeMark1d, DeMark3d, Guo, MarkZAMP}. Furthermore, the technique developed in this paper can be used to deal with subsonic flow of nonisentropoic Euler-Poisson system with nonzero vorticity. This case will be discussed in the upcoming papers.

\subsection{One dimensional subsonic flow}\label{appendix-pf}
Before we proceed further to prove Theorem \ref{theorem1}, we first prove Proposition \ref{lemma-2-1}.
\begin{proof}[Proof of Proposition \ref{lemma-2-1}]
For simplicity, we regard a one dimensional solution $(\vphi,\Phi)(x_1,\cdots,x_n)$ in $\N$ as functions $(\vphi,\Phi)(x_n)$ on the interval $(0,L)$. Note that any functions $(\vphi,\Phi)(x_n)$ satisfy the boundary conditions \eqref{2-b9-2} on $\Gam_w$.
Suppose that $(\vphi(x_n),\Phi(x_n))$ is a one-dimensional solution to the system of \eqref{2-b2} and \eqref{2-b3d} in $\N$, and let us set $u:=\der_{x_n}\vphi$ and $E=\der_{x_n}\Phi$. Then, for $\rho$ defined by \eqref{2-b4}, $(\rho, u, E)$ satisfy
\begin{equation}
\label{2-c2}
\begin{split}
&(\rho u)'=0\\
&(\rho  u^2+p(\rho))'=\rho E\\
&E'=\rho-b_0
\end{split}
\end{equation}
in the interval $(0,L)$ where $'$ is the differentiation with respect to $x_n$. The first equation in \eqref{2-c2} implies that
\begin{equation}
\label{2-c3}
\rho u \equiv J_0\quad\tx{on}\quad(0,L)
\end{equation}
for a constant $J_0>0$. Then by repeating the analysis in \cite{LuoXin} for the ODE system \eqref{2-c2}, we can state as follows:

\emph{Fix constants $b_0>0$ and $L>0$. For any given constant $J_0>0$, there exists a nonempty set $\mathfrak{P}_1(J_0)\subset \R^+\times \R$ so that for any $(\rho_0, E_0)\in \mathfrak{P}_1(J_0)$, the initial value problem}
\begin{equation}
\label{ibp}
\begin{cases}
u=\frac{J_0}{\rho}\\
(\rho u^2+p(\rho))'=\rho E\\
E'=\rho-b_0
\end{cases}\tx{on}\;\;(0,L),\quad
\begin{cases}
\rho(0)=\rho_0\\
E(0)=E_0
\end{cases}
\end{equation}
{\emph{has a unique smooth solution $(\rho,u,E)(x_n)$ on the interval $[0,L]$ with satisfying }}
\begin{equation}
\label{2-a7}
\rho>0\quad\tx{and}\quad p'(\rho)-u^2\ge \nu_1>0\quad\tx{on}\quad[0,L]
\end{equation}
{\emph{for some positive constant $\nu_1$ where the choice of $\nu_1$ depends on $\rho_0, E_0, J_0$ and $L$.} }

Hereafter, $\mathfrak{P}_1(J_0)$ denotes the maximal set of $(\rho_0,E_0)$ for which \eqref{ibp} has a unique smooth solution $(\rho,u,E)$ satisfying \eqref{2-a7}.

For each $J_0>0$ and $(\rho_0,E_0)\in \mathfrak{P}_1(J_0)$, we claim that there exists unique corresponding boundary data $(\bar\Phi_{en,0},\msB_{0,0}, p_{ex,0})$ for boundary conditions \eqref{2-b8}--\eqref{bcpressure}. Fix $J_0>0$ and $(\rho_0,E_0)\in\mathfrak{P}_1(J_0)$. The first equation of \eqref{ibp} yields $u(0)=\frac{J_0}{\rho_0}$. By substituting this into \eqref{2-a4}, we get $\msB(0)=\frac 12(\frac{J_0}{\rho_0})^2+\int_{k_0}^{\rho_0}\frac{p'(\varrho)}{\varrho}d\varrho$. Therefore
$\Phi(0)=\msB(0)=\frac12(\frac{J_0}{\rho_0})^2+\int_{k_0}^{\rho_0}\frac{p'(\varrho)}{\varrho}d\varrho$ is obtained from the equation $\msB-\Phi=0$.
Using this equation again gives $\Phi(L)=\msB_{0,0}$ and $\Phi(0)=\bar\Phi_{en,0}+\msB_{0,0}$. Therefore, we get the following equation for $\bar\Phi_{en,0}$ and $\msB_{0,0}$
\begin{equation}
\label{1d-eqn1}
\bar\Phi_{en,0}+\msB_{0,0}=\frac12\left(\frac{J_0}{\rho_0}\right)^2+\int_{k_0}^{\rho_0}\frac{p'(\varrho)}{\varrho}d\varrho.
\end{equation}
For the solution $(\rho, E)$ to \eqref{ibp}, set
\begin{equation}
\label{1d-eqn2}
p_{ex,0}=p(\rho(L)).
\end{equation}
Substituting $\rho=\rho(L)$ and $u(=|\nabla\vphi|)=\frac{J_0}{\rho(L)}$ into \eqref{2-a4} yields
\begin{equation}
\label{1d-eqn3}
\msB_{0,0}=\msB(L)=\frac 12\left(\frac{J_0}{\rho(L)}\right)^2+\int_{k_0}^{\rho(L)}\frac{p'(\varrho)}{\varrho}d\varrho.
\end{equation}
By \eqref{1d-eqn1} and \eqref{1d-eqn3}, the value of $\bar\Phi_{en,0}$ is uniquely determined by $(\rho_0,E_0)$ and $L$. For such $(\bar\Phi_{en,0},\msB_{0,0}, p_{ex,0})$,  define two functions $\vphi$ and $\Phi$ by
\begin{equation}
\label{2-b6}
\vphi(x_n):=\int_0^{x_n} u(t) dt,\quad \Phi(x_n):=\msB_{0,0}+\bar{\Phi}_{en,0}+\int_{0}^{x_n}E(t) dt
\end{equation}
for the solution $(\rho, E)$ to \eqref{ibp}. Then the pair of $(\vphi,\Phi)$ becomes a $C^2$ solution of \eqref{2-b2} and \eqref{2-b3d} in $\N$ with the boundary conditions \eqref{2-b8}--\eqref{bcvphi}. Moreover, by \eqref{2-a7},  we have
\begin{equation}
\label{ellipticity}
p'(\rho(\Phi,|\nabla\vphi|^2))-|\nabla\vphi|^2\ge \nu_1 \quad\tx{in}\quad\ol{\N}
\end{equation}
for $\nu_1>0$ same as in \eqref{2-a7}.
This proves the existence of nonempty parameter set $\mathfrak{P}_0$. Next, we prove the uniqueness of one-dimensional solution for fixed $(\bar\Phi_{en,0}, \msB_{0,0}, p_{ex,0})\in\mathfrak{P}_0$.

Fix the boundary data $(\bar\Phi_{en,0}, \msB_{0,0}, p_{ex,0})\in\mathfrak{P}_0$, and let $(\vphi^{(1)}, \Phi^{(1)})$ and $(\vphi^{(2)}, \Phi^{(2)})$ be two one-dimensional solutions of \eqref{2-b2} and \eqref{2-b3d} in $\N$ with the boundary conditions \eqref{2-b8}--\eqref{bcvphi} and satisfying \eqref{2-b5}. For each $j=1,2$, let $\rho^{(j)}$ be defined by \eqref{2-b4} for $(\vphi^{(j)}, \Phi^{(j)})$. By \eqref{2-a5}, for any given $p_{ex,0}>0$, there exists unique $\rho_{ex,0}>0$ satisfying $p(\rho_{ex,0})=p_{ex,0}$. Then each $u^{(j)}=(\vphi^{(j)})'$ satisfies
\begin{equation}
\label{2-c7}
\frac 12 (u^{(j)}(L))^2=\msB_{0,0}-\int_{k_0}^{\rho_{ex,0}}\frac{p'(\varrho)}{\varrho}d\varrho
\quad\tx{for}\quad j=1,2.
\end{equation}
This fixes unique value of $J_0=\rho_{ex,0}u^{(j)}(L)$ for $j=1,2$ in the equation \eqref{2-c3} so that each $\Phi^{(j)}$ can be expressed as
\begin{equation}
\label{2-c4}
\Phi^{(j)}(x_n)=\frac 12\Bigl(\frac{J_0}{\rho^{(j)}(x_n)}\Bigr)^2
+\int_{k_0}^{\rho^{(j)}(x_n)}\frac{p'(\varrho)}{\varrho}d\varrho=:H(\rho^{(j)}(x_n))\quad\tx{on}\quad[0,L].
\end{equation}
Since $p''(\rho)>0$ for $\rho>0$, there exists unique $\rho_s>0$ depending on $J_0$ so that we have
\begin{equation}
\label{2-c8}
H'(\rho)=\frac{1}{\rho^3}(\rho^2p'(\rho)-J_0^2)
\begin{cases}
<0&\tx{for}\;\;\rho<\rho_s\\
=0&\tx{at}\;\;\rho=\rho_s\\
>0&\tx{for}\;\;\rho>\rho_s
\end{cases}.
\end{equation}
It follow from \eqref{2-c8} that
\begin{equation*}
\int_0^1 H'(t\rho^{(1)}+(1-t)\rho^{(2)})dt>0\quad\tx{on}\quad [0,L]
\end{equation*}
because $\rho^{(j)}>\rho_s$ on $[0,L]$ for $j=1,2$ by \eqref{2-b5}.
Using \eqref{2-b3d}, \eqref{pdiff} and \eqref{2-c4} with $\Phi(L)=\msB_{0,0}$ and $\bar\Phi_{ex}=0$, one can easily show that $\Phi^{(1)}-\Phi^{(2)}$ satisfies
\begin{equation*}
(\Phi^{(1)}-\Phi^{(2)})''-
\frac{\Phi^{(1)}-\Phi^{(2)}}{\int_0^1H'(t\rho^{(1)}+(1-t)\rho^{(2)})dt}=0
\quad\tx{on}\quad(0,L),
\end{equation*}
and the boundary conditions $\Phi^{(1)}-\Phi^{(2)}=0$ at $x_n=0$ and $L$. Then the maximum principle for elliptic equations implies $\Phi^{(1)}=\Phi^{(2)}$ on $[0,L]$, and this yields $\rho^{(1)}=\rho^{(2)}$ on $[0,L]$. So we get $\vphi^{(1)}=\vphi^{(2)}$ on $[0,L]$.

For any given $(\bar\Phi_{en,0}, \msB_{0,0}, p_{ex,0})\in\mathfrak{P}_0$, we have shown that there exists a unique one-dimensional subsonic solution $(\vphi,\Phi)$. Thus, for the density $\rho$ defined by \eqref{2-b4} from $(\vphi,\Phi)$, we get $(\rho_0,E_0)=(\rho(\Phi(0),|\vphi'(0)|^2),\Phi'(0))\in\mathfrak{P}_1(J_0)$ where $J_0$ is uniquely determined from $(\bar\Phi_{en,0}, \msB_{0,0}, p_{ex,0})$ through \eqref{2-c7}. Hence we can find the maximal set of parameter set $\mathfrak{P}_0$ satisfying the statement of Proposition \ref{lemma-2-1} by collecting all parameters $(\bar\Phi_{en,0}, \msB_{0,0}, p_{ex,0})$  corresponding to $(\rho_0,E_0)\in \underset{J_0>0}{\cup}\mathfrak{P}_1(J_0)$.
\end{proof}

\begin{remark}
\label{remark-1a}
The proof of Proposition \ref{lemma-2-1} shows that, for any given $(\bar\Phi_{en,0}, \msB_{0,0}, p_{ex,0})\in \mathfrak{P}_0$, the constant $J_0(=\rho u)$ is uniquely determined. Therefore, one can use the first two equations in \eqref{2-c2} to rewrite the last equation in \eqref{2-c2} as
\begin{equation}\label{1deq}
\left(\frac{1}{\rho}\left(\frac{J_0^2}{\rho}+p(\rho)\right)'\right)'=\rho -b_0.
\end{equation}
By \eqref{2-a7}, one can also find a unique constant $\rho_{en,0}>0$ satisfying the equation \eqref{1d-eqn1} with $\rho_0=\rho_{en,0}$. Therefore, the boundary value problem in Proposition \ref{lemma-2-1} can be considered as a boundary value problem for the equation \eqref{1deq} with the boundary conditions
\begin{equation*}
%\label{iv}
\rho(0)=\rho_{en,0},\quad \rho(L)=\rho_{ex,0}.
\end{equation*}
And, this is equivalent to fix the pressure at both ends of the nozzle $\N$ by \eqref{2-a8}.
\end{remark}

The rest of the paper is organized as follows. In Section \ref{Section-linearization}, we rewrite \eqref{nbvp-1}--\eqref{nbvp-4} as a nonlinear boundary value problem for $(\psi,\Psi):=(\vphi-\vphi_0, \Phi-\Phi_0)$, and set an iteration scheme to solve the nonlinear boundary value problem for $(\psi,\Psi)$. In Section \ref{section-3}, we achieve a priori $H^1$ estimate of weak solutions to a linear boundary value problem containing a second order elliptic system. By using this estimate, we prove the unique existence of weak solutions and establish a priori $C^{\alp}$ estimate of the weak solutions up to the boundary of $\N$. Then a priori $C^{1, \alpha}$ estimates and weighted $C^{2,\alp}$ estimates of the weak solutions are obtained, and these estimates lead to prove the unique existence of $C^2$ solutions to the linear boundary value problem. It should be emphasized that the $H^1$ estimate of weak solutions for second order elliptic system is the key ingredient to get the main result of this paper. Section \ref{section4} is devoted to prove Theorem \ref{theorem1} by applying the a priori estimates given in Section \ref{section-3} and the contraction mapping principle. In Appendix \ref{section-nonconstant-b}, the existence of one dimensional subsonic solutions to the Euler-Poisson system with variable background charge is proved. The last appendix gives an outline of proof for stability of subsonic flows under small perturbations of the boundaries of the nozzle $\N$.

\section{Linearized boundary value problem and Iteration scheme}
\label{Section-linearization}
For the rest of paper, let $\N$, $b_0$, $L$ be as in Theorem \ref{theorem1}. And we fix $(\bar \Phi_{en,0},\msB_{0,0},p_{ex,0})\in \mathfrak{P}_0$, and let $(\vphi_0,\Phi_0)$ be the corresponding background solution. Let $b$, $(\bar{\Phi}_{en}, \bar{\Phi}_{ex},\pex)$ satisfy the estimates \eqref{data} for $\sigma\in(0,\bar{\sigma})$ with $\bar{\sigma}$ to be determined later.
\subsection{Linearization of the equations in \eqref{nbvp-1}}
\label{subsec-lin-eqns}

For $(z, \q)=(z,q_1,\cdots, q_n)\in \R\times \R^n$ and for $\rho$ defined by \eqref{2-b4}, set
\begin{equation}
\label{3-a1}
{\bf A}(z,\q)=(A_1,\cdots, A_n)(z,\q)=\rho(z,|{\bf q}|^2){\bf q},
\quad
B(z,\q)=\rho(z,|{\bf q}|^2).
\end{equation}
Fix a constant $\eps_0>0$.
For any $(z,\q)$ satisfying $z-\frac 12|\q|^2\ge h(\eps_0)$ for $h$ defined by \eqref{rho-eqn}, ${\bf A}(z,{\bf q})$ and $B(z,{\bf q})$ in \eqref{3-a1} are well-defined, and they are differentiable with respect to $z$ and ${\bf q}$. Furthermore, one can directly check that
\begin{equation}
\label{3-a2}
\begin{split}
&\der_z A_i(z,\q)=\der_zB(z,\q)q_i,\quad  \der_zB(z,\q)=\frac{B(z,\q)}{p'(B(z,\q))},\\
&\der_{q_j}A_i(z,\q)=B(z,\q)\Bigl(\delta_{ij}-\frac{q_iq_j}{p'(B(z,\q))}\Bigr),\\
& \der_{q_j}B(z,\q)=-\frac{B(z,\q)q_j}{p'(B(z,\q))}
\end{split}
\end{equation}
for $i,j=1,\cdots, n$, where $\delta_{ij}=1$ for $i=j$ and $\delta_{ij}=0$ otherwise.

Suppose that $(\vphi,\Phi)\in [C^2(\N)]^2$ satisfy the equations in \eqref{nbvp-1} in $\N$, and set
\begin{equation}
\label{3-f4}
(\psi,\Psi):=(\vphi,\Phi)-(\vphi_0,\Phi_0).
\end{equation}
Since $(\vphi_0,\Phi_0)$ satisfy \eqref{nbvp-1} as well, by subtracting the equation $\Div(\rho(\Phi_0,|\nabla\vphi_0|^2)\nabla\vphi_0)=0$ from the first equation in \eqref{nbvp-1}, one has
\begin{equation}
\label{3-a3}
\begin{aligned}
L_1(\psi,\Psi):= & \Div\left(\sum_{j=1}^n\der_{q_j}A_i(\Phi_0, D\vphi_0)\der_j\psi+\Psi\der_z{\bf A}(\Phi_0,D\vphi_0)\right)\\
=& \Div{{\bf F}(\rx,\Psi,D\psi)}\;\;\tx{in}\;\;\N
\end{aligned}
\end{equation}
for ${\bf F}=(F_1,\cdots, F_n)$ given by
\begin{equation}
\label{3-a5}
-F_i(\rx,z,{{\bf q}})=\int_0^1 z\bigl[\der_zA_i(\Phi_0+\til z, D\vphi_0+{\til{\bf q}})\bigr]_{(\til z,{\til{\bf q}})=(0,\bm 0)}^{(t z,{\bf q})}+ q_j\bigl[\der_{q_j}A_i(\Phi_0,D\vphi_0+\til{\bf q})\bigr]_{\til{\bf q}=\bm 0}^{t {\bf q}}dt
\end{equation}
where $[G(X)]_{X=a}^b:=G(b)-G(a)$ with $\Phi_0$ and $D\vphi_0$ evaluated at $\rx\in\N$. Here, $D$ denotes $(\der_{x_1},\cdots, \der_{x_n})$, and each $\der_j$ denotes $\der_{x_j}$ for $j=1,\cdots, n$.

Next, subtracting the equation $\Delta \Phi_0=\rho(\Phi_0,|\nabla\vphi_0|^2)-b_0$ from the second equation in \eqref{nbvp-1} gives
\begin{equation}
\label{3-a4}
L_2(\psi,\Psi):=\Delta \Psi-\der_zB(\Phi_0,D\vphi_0)\Psi-\der_{{\bf q}}B(\Phi_0,D\vphi_0)\cdot D\psi=f(\rx,\Psi,D\psi)+b_0-b(\rx)
\end{equation}
for $f$ given by
\begin{equation}
\label{3-a6}
\begin{split}
f(\rx,z,{{\bf q}})
=\int_0^1z\bigl[\der_z B(\Phi_0+\til z, D\vphi_0+\til{\bf q})\bigr]_{(\til z,\til{\bf q})=(0,\bm 0)}^{(tz,\q)}+ q_j\bigl[\der_{q_j}B(\Phi_0,D\vphi_0+\til \q)\bigr]_{\til q=\bm 0}^{t\q} dt.
\end{split}
\end{equation}
There is a constant $\delta_1\in(0,1)$ depending on the data such that each $F_i(\rx,z,{\bf q})$ and $f(\rx,z, {\bf q})$ are well defined for all $(x,z,{\bf q})\in \mcl{D}_{3\delta_1}$ for
\begin{equation}
\label{delta1}
\mcl{D}_{3\delta_1}:=\N\times \{(z, {\bf q})\in  \R\times \R^n: |z|+|{\bf q}|< 3\delta_1\}.
\end{equation}
Then, the following lemma can be easily checked.
\begin{lemma}
\label{lemma-3-4}
Let $F_i$ and $f$ be defined by \eqref{3-a5} and \eqref{3-a6}, respectively. Then,
\begin{itemize}
\item[(a)] there exists a constant $C$ depending only on the data such that
\begin{equation*}
\begin{split}
&{\bf F}(\rx,0,{\bf 0})={\bf 0},\quad f(\rx,0,{\bf 0})=0,\\
&|D_{(z,{\q})}F_i(x,z,{\bf q})|\le C(|z|+|{\bf q}|),\;\;
|D_{(z,{\q})}f(x,z,{\bf q})|\le C(|z|+|{\bf q}|)
\end{split}
\end{equation*}
for all $(x, z,{\bf q})\in \mcl{D}_{2\delta_1}$ and $i=1,\cdots,n$;
\item[(b)] for any integer $k\ge 2$, there exists a constant $C_k$ depending only on the data and $k$ such that
    \begin{equation*}
    |D^k_{({z},{\bf q})}F_i(x,z,{\bf q})|\le C_k,\;\;
    |D^k_{(z,{\bf q})}f(x,{z},{\bf q})|\le C_k
    \end{equation*}
    for all $(x, z,{\bf q})\in \mcl{D}_{2\delta_1}$. Here, $D^k_{(z, \q)}$  denotes the $k$-th order derivatives with respect to $z$ and ${\bf q}$.
    \end{itemize}

\end{lemma}

\begin{lemma}
\label{lemma-3-1}
Let ${\bf A}(z,{\bf q})$ and $B(z,{\bf q})$ be as in \eqref{3-a1}.
\begin{itemize}
\item[(a)] The matrix $[a_{ij}]_{i,j=1}^n:=[\der_{q_j}A_i(\Phi_0,D\vphi_0)]_{i,j=1}^n$ is a strictly positive diagonal matrix in $\N$, and there exits a constant $\lambda>0$ satisfying
    \begin{equation}
    \label{3-b1}
    \lambda { I_n}\le [a_{ij}]_{i,j=1}^n\le \frac{1}{\lambda}{ I_n},
    \end{equation}
    and such $\lambda$ depends only on the data.
\item[(b)] Each $a_{ij}$ is smooth in $\N$. Also, there exists a constant $C_k>0$ depending on the data and $k$ to satisfy $\|a_{ij}\|_{C^k(\ol{\N})}\le C_k$ for all $i,j=1,\cdots, n$ and any nonnegative integer $k$.
\item[(c)] Let $h(\rho)$ be defined as in \eqref{rho-eqn}. For any constant $\eps_0>0$ and any $(z,\q)\in \R\times \R^n$ satisfying $z-\frac 12|\q|^2\ge h(\eps_0)$, the equality
\begin{equation}
\label{3-b2}
\der_{z}{\bf A}(z,\q)+\der_{\bf q}B(z,\q)=0
\end{equation}
holds.

\end{itemize}
\begin{proof}
From \eqref{2-b6} and \eqref{3-a2}, it is clear that $[a_{ij}]_{i,j=1}^n$ is a diagonal matrix with
\begin{equation}
\label{3-b3}
a_{ii}=
\begin{cases}
B(\Phi_0, D\vphi_0)&\tx{for}\;\;i<n\\
B(\Phi_0,D\vphi_0)\bigl(1-\frac{|D\vphi_0|^2}{p'(B(\Phi_0,D\vphi_0))}\bigr)&\tx{for}\;\;i=n.
\end{cases}
\end{equation}
Furthermore, \eqref{background-sol-prop} implies that there exists a constant $\lambda>0$ satisfying \eqref{3-b1}. So (a) is proved. (b) is already stated in Definition \ref{def-background-sol}. Finally, (c) can be easily checked from \eqref{3-a2}.
\end{proof}
\end{lemma}
Lemma \ref{lemma-3-1}(c) is crucial to get $H^1$ estimates of weak solutions to a linear boundary value problem associated with two elliptic operators $L_1$ and $L_2$ defined by \eqref{3-a3} and \eqref{3-a4}, respectively.
\subsection{Boundary conditions for $(\psi,\Psi)$}
\label{subsec-lin-bcs}
Given $\bar\Phi_{en}, \bar\Phi_{ex}, \pex$ and $\msB_0$, suppose that $(\vphi,\Phi)$ satisfy the boundary conditions \eqref{nbvp-2}--\eqref{nbvp-4}. Then $(\psi,\Psi)$ defined by \eqref{3-f4} satisfy
\begin{equation}
\label{3-f6}
\begin{split}
&\psi=0\quad\tx{on}\quad \Gam_0,\\
&\Psi=\begin{cases}
(\msB_0-\msB_{0,0})+(\bar\Phi_{en}-\bar\Phi_{en,0})=:\Psi_{en}&\tx{on}\;\;\Gam_0\\
(\msB_0-\msB_{0,0})+\bar\Phi_{ex}=:\Psi_{ex} &\tx{on}\;\;\Gam_L
\end{cases},\\
&\der_{{\bf n}_w}\psi=\der_{{\bf n}_w}\Psi=0\quad\tx{on}\quad \Gam_w
\end{split}
\end{equation}
\begin{equation}
\label{3-f7}
p(B(\Phi_0+\Psi,\nabla\vphi_0+\nabla\psi))-p(B(\Phi_0,\nabla\vphi_0))
=\pex-p_{ex,0}\quad\tx{on}\;\;\Gam_{ex}
\end{equation}
for $B(z,\q)$ defined by \eqref{3-a1}.
To simplify notations, set
\begin{equation*}
\rho:=B(\Phi_0+\Psi,\nabla\vphi_0+\nabla\psi),\quad \rho_{bg}:=B(\Phi_0,\nabla\vphi_0).
\end{equation*}
There exists a constant $\delta_2>0$ depending only on the data so that if
\begin{equation}
\label{delta2}
\|\Psi\|_{C^0(\ol{\N})}+\|\psi\|_{C^1(\ol{\N})}\le 4\delta_2,
\end{equation}
then
\begin{equation}
\label{lbd-density}
B(\Phi_0+\Psi, \nabla(\vphi_0+\psi))=h^{-1}(\Phi_0+\Psi-\frac 12|\nabla(\vphi_0+\psi)|^2)\ge \frac 12\nu_0>0
\end{equation}
holds in $\N$ for the constant $\nu_0$ in \eqref{background-sol-prop}. Then the left-hand side of \eqref{3-f7} is rewritten as $(\rho-\rho_{bg})\int_0^1p'(t\rho+(1-t)\rho_{bg})dt$, and $\rho-\rho_{bg}$ can be written as
\begin{equation*}
%\label{3-f8}
\begin{split}
&(\rho-\rho_{bg})\\
&=\Psi\int_0^1 B_z(\Phi_0+t\Psi,\nabla\vphi_0+t\nabla\psi)dt
+\nabla\psi\cdot\int_0^1 B_{\q}(\Phi_0+t\Psi,\nabla\vphi_0+t\nabla\psi)dt.
\end{split}
\end{equation*}
It follows that the boundary condition \eqref{3-f7} can be rewritten as
\begin{equation}
\label{3-f9}
B_{\q}(\Phi_0,\nabla\vphi_0)\cdot\nabla\psi
=\frac{\pex-p_{ex,0}}{\int_0^1p'(t\rho+(1-t)\rho_{bg})dt}+\hat g_2(\rx, \Psi,\nabla\psi)\quad\tx{on}\quad\Gam_L
\end{equation}
for
\begin{equation}
\label{3-g6}
\hat g_2(x,\Psi,\nabla\psi)=-\int_0^1\Psi B_z(\Phi_0+t\Psi,\nabla\vphi_0+t\nabla\psi)
+\nabla\psi\cdot [B_{\q}(\Phi_0+t\hat z,\nabla\vphi_0+t\hat{\q})]_{(\hat z,\hat{\q})=(0,0)}^{(\Psi,\nabla\psi)}dt.
\end{equation}
By \eqref{3-f6}, we can substitue $\Psi=\Psi_{ex}$ into \eqref{3-f9} to get a nonlinear boundary condition for $\psi$:
\begin{equation}
\label{3-g7}
B_{\q}(\Phi_0,\nabla\vphi_0)\cdot\nabla\psi
=g(\rx,\nabla\psi,\pex,\Psi_{ex})\quad\tx{on}\quad\Gam_L
\end{equation}
for
\begin{equation}
\label{3-g8}
\begin{split}
&g(\rx,\q,\pex,\Psi_{ex})=\frac{\pex-p_{ex,0}}{g_1(\rx,\q,\pex,\Psi_{ex})}
+\hat g_2(\rx,\Psi_{ex},\q)\\
&g_1(\rx,\q,\pex,\Psi_{ex})=\int_0^1p'(tB(\Phi_0+\Psi_{ex},\nabla\vphi_0+\q)
+(1-t)B(\Phi_0,\nabla\vphi_0))dt,
\end{split}
\end{equation}
where $\Phi_0, \vphi_0,$ and $\Psi_{ex}$ are evaluated at $\rx\in\Gam_L$.
\begin{lemma}
\label{lemma-3-5}
Let $\delta_2$ be as in \eqref{delta2}. Under the assumption \eqref{data}, if $4\sigma\le \delta_2$, then the following properties hold:
\begin{itemize}
\item[(a)] there exists a constant $C$ depending only on the data so that
    \begin{equation*}
    |D_{\q}g(\rx,\q,\pex,\Psi_{ex})|\le C(|\pex(\rx)-p_{ex,0}|+|\Psi_{ex}(\rx)|+|\q|)
    \end{equation*}
    for all $(\rx,\q)\in\ol{\Gam_L}\times\{\q\in\R^n:|\q|<2\delta_2\}$;
\item[(b)] for any integer $k\ge 2$, there exists a constant $C_k$ depending only on the data and $k$ so that
\begin{equation*}
|D^k_{\q}g(\rx,\q,\pex,\Psi_{ex})|\le C_k
\end{equation*}
for all $(\rx,\q)\in\ol{\Gam_L}\times\{\q\in\R^n:|\q|<2\delta_2\}$.
\end{itemize}
\end{lemma}
Lemma \ref{lemma-3-5} is a direct consequence of \eqref{lbd-density} and \eqref{3-g8} so we skip the proof.

For later use, we note that the boundary conditions $\der_{{\bf n}_w}\psi=0$ on $\Gam_w$ and \eqref{3-g7} on $\Gam_L$ can be rewritten as conormal boundary conditions. On $\Gam_w$, the inward unit normal vector field ${\bf n}_w$ on $\Gam_w$ satisfies ${\bf n}_w\cdot (\underset{(n-1)}{\underbrace{0,\cdots, 0}}, 1)=0$. So, by \eqref{3-b3}, $\der_{{\bf n}_w}\psi=0$ is equivalent to
\begin{equation}
\label{conormal-1}
(\sum_{j=1}^na_{ij}\der_j\psi)\cdot{\bf n}_w=0\quad\tx{on}\quad\Gam_w.
\end{equation}
It follows from \eqref{3-a2} that $B_{\q}(\Phi_0,\nabla\vphi_0)=(\underset{(n-1)}{\underbrace{0,\cdots, 0}}, -\frac{J_0}{p'(B(\Phi_0, \nabla\vphi_0))})$ for $J_0=B(\Phi_0,\nabla\vphi_0)\der_n\vphi_0$ where $\Phi_0, \vphi_0$ and $\der_{n}\vphi_0$ are evaluated on $\Gam_L$. Combining this with \eqref{3-g7} gives
\begin{equation}
\label{conormal-2}
(\sum_{j=1}^na_{ij}\der_j\psi)\cdot{\bf n}_L=
\frac{a_{nn}(\rx)p'\left(B(\Phi_0,\nabla\vphi_0)\right)}{J_0}g(\rx,\nabla\psi,\pex,\Psi_{ex})
\quad\tx{on}\quad \Gam_L
\end{equation}
for the inward unit normal ${\bf n}_L=(\underset{(n-1)}{\underbrace{0,\cdots, 0}},-1)$ on $\Gam_L$.

\subsection{Iteration scheme}
\label{subsec-iteration}
Suppose that $\bar\sigma<1$ in Theorem \ref{theorem1}.
In \S \ref{subsec-lin-eqns} and \S \ref{subsec-lin-bcs}, we have seen that $(\vphi,\Phi)\in [C^1(\ol{\N})\cap C^2(\N)]^2$ solve \eqref{nbvp-1}--\eqref{nbvp-4} if and only if $(\psi,\Psi)=(\vphi-\vphi_0,\Phi-\Phi_0)$ satisfy the equations \eqref{3-a3} and \eqref{3-a4} in $\N$ and the boundary conditions \eqref{3-f6} and \eqref{3-g7} provided that
\begin{equation}
\label{delta-min}
\|\Psi\|_{C^0(\ol{\N})}+\|\psi\|_{C^1(\ol{\N})}\le 2\min\{\delta_1,\delta_2\}
\end{equation}
for $\delta_1$ and $\delta_2$ in \eqref{delta1} and \eqref{delta2}, respectively.

Given $\alp\in(0,1)$, we define an iteration set $\mcl{K}_M$ as follows:
\begin{equation}
\label{3-g1}
\mcl{K}_M:=\mcl{K}^{(1)}_M\times \mcl{K}^{(2)}_M
\end{equation}
where $\mcl{K}^{(1)}$ and $\mcl{K}^{(2)}$ are given by
\begin{equation*}
\begin{split}
&\mcl{K}^{(1)}_M:=\{\tpsi\in C^{2,\alp}_{(-1-\alp,\Gam_0\cup\Gam_L)}({\N}): \til{\psi}=0\;\tx{on}\;\Gam_0,\;
\|{\tpsi}\|_{2,\alp,\N}^{(-1-\alp,\corners)}\le \frac M2 \sigma\}\\
&\mcl{K}_M^{(2)}:=\{\tPsi\in C^{2,\alp}_{(-1-\alp,\Gam_0\cup\Gam_L)}({\N}):
\|{\tPsi}\|_{2,\alp}^{(-1-\alp,\corners)}\le \frac{M}{2}\sigma\}
\end{split}
\end{equation*}
for constants $M>1$ and $\sigma\in(0,1)$ to be determined later. If $M\sigma\le 2\min\{\delta_1,\delta_2\}$, then ${\bf F}(\rx,\tPsi, D\tpsi)$, $f(\rx,\tPsi,D\tpsi)$ are well defined in $\N$ for all $(\til{\psi}, \tPsi)\in\mcl{K}_M$, and $g(\rx,\nabla\tpsi,\pex,\Psi_{ex})$ is well defined on $\Gam_L$ for ${\bf F}, f, g$ defined by \eqref{3-a5}, \eqref{3-a6}, \eqref{3-g8}, respectively. Consider the following linear boundary value problem for $(\hat\psi,\hat\Psi)$
\begin{align}
\label{3-g2}
&\begin{cases}
L_1(\hat\psi,\hat\Psi)=\Div {\bf F}(\rx,\tPsi,D\tpsi)\\
L_2(\hat\psi,\hat\Psi)=f(\rx,\tPsi, D\tpsi)
\end{cases}\;\;\tx{in}\;\;\N
\end{align}
with boundary conditions \eqref{3-f6}(replacing $(\psi,\Psi)$ on the left-hand side by $(\hat{\psi},\hat{\Psi})$) and
\begin{align}
\label{3-g4}
B_{\q}(\Phi_0,\nabla\vphi_0)\cdot\nabla\hat\psi=
g(\rx,\nabla\tpsi,\pex,\Psi_{ex})\quad\tx{on}\quad\Gam_L,
\end{align}
where $L_1$ and $L_2$ are defined in \eqref{3-a3} and \eqref{3-a4}, respectively.
If  the linear elliptic system \eqref{3-g2} with boundary conditions \eqref{3-f6} and \eqref{3-g4} has a unique solution $(\hat\psi,\hat\Psi)\in [C^1(\ol{\N})\cap C^2(\N)]^2$, then an iteration mapping $\mathfrak{I}$ can be defined by
\begin{equation}
\label{3-g5}
\mathfrak{I}:(\tpsi,\tPsi)(\in\mcl{K}_M)\mapsto (\hat\psi,\hat{\Psi}).
\end{equation}
It is easy to see that $(\psi^*,\Psi^*)\in\mcl{K}_M$ is a fixed point of $\mathfrak{I}$ if and only if $(\vphi^*,\Phi^*)=(\vphi_0,\Phi_0)+(\psi^*,\Psi^*)$ is a solution to the nonlinear boundary value problem of \eqref{nbvp-1}--\eqref{nbvp-4}. So it suffices to prove unique existence of fixed point of $\mathfrak{I}$ in $\mcl{K}_M$ for the proof of Theorem \ref{theorem1}.
\begin{proposition}
\label{theorem2}
Let $\N$ be as in \eqref{domain}. Fix $b_0>0$ and $L>0$, and let the parameter set $\mathfrak{P}_0$ be as in Proposition \ref{lemma-2-1}. Given $(\bar\Phi_{en,0}, \msB_{0,0}, p_{ex,0})\in\mathfrak{P}_0$, let $(\vphi_0,\Phi_0)$ be the corresponding background solution. Assume that $b,(\bar\Phi_{en},\bar\Phi_{ex},p_{ex})$ and $\msB_0$ satisfy \eqref{data} and \eqref{compatibility}. Then, for any given $\alp\in(0,1)$, there exist constants $M>1$ and $\sigma_1\in(0,1)$ depending on the data and $\alp$ such that wherever $\sigma\in(0,
\sigma_1]$ in \eqref{data}, the iteration mapping $\mathfrak{I}$ defined by \eqref{3-g5} has a unique $(\psi^*,\Psi^*)$ in $\mcl{K}_M$ satisfying
\begin{align}
&\mathfrak{I}(\psi^*,\Psi^*)=(\psi^*,\Psi^*),\notag\\
\label{subsonicity}
&p'(\rho(\Phi_0+\psi^*,|\nabla\vphi_0+\nabla\psi^*|^2))
-|\nabla\vphi_0+\nabla\psi^*|^2\ge \nu_1
\end{align}
for a constant $\nu_1>0$.
\end{proposition}
Theorem \ref{theorem1} easily follows from Proposition \ref{theorem2} by choosing $\bar{\sigma}=\sigma_1$. So the rest of the paper is devoted to prove Proposition \ref{theorem2}.

\section{Gradient estimates for elliptic system and mixed boundary value problem}
The following a priori estimate is essential  to prove Proposition \ref{theorem2}.
\label{section-3}
\begin{proposition}
\label{proposition-3-1}
Fix $\alp\in(0,1)$ and functions ${\bf F}\in C^{1,\alp}_{(-\alp,\Gam_0\cup\Gam_L)}(\N,\mathbb{R}^n)$, $f\in C^{\alp}(\ol{\N})$, $W_{en}\in C^{2,\alp}(\ol{\Gam_0})$, $W_{ex}\in C^{2,\alp}(\ol{\Gam_L})$ and a function $g\in C^{\alp}(\ol{\Gam_L})$ satisfying the following properties:
\begin{align}
\label{F-condition}
&{\bf F}\cdot \hat {\bf e}_j=0\quad\tx{on}\quad\Gam_0\;\;\tx{for}\;\;j=1,\cdots,n-1;\\
\label{normalization}
&\|{\bf F}\|_{1,\alp,\N}^{(-\alp,\Gam_0\cup\Gam_L)}+\| f\|_{\alp,\N}+\|g\|_{\alp,\Gam_L}
+\|W_{en}\|_{2,\alp,\Gam_0}+\|W_{ex}\|_{2,\alp,\Gam_L}\le 1
\end{align}
where $\hat{\bf e}_j$ denotes the constant vector $(0,\cdots,0,\underset{j^{th}}{\underbrace{1}},0,\cdots,0)$ for each $j=1,\cdots,n-1$;
\begin{equation}
\label{compatibility-linear}
\der_{{\bf n}_w}W_{en}=0\quad\tx{on}\;\;\ol{\Gam_0}\cap\ol{\Gam_w},\quad
\der_{{\bf n}_w}W_{ex}=0\quad\tx{on}\;\;\ol{\Gam_L}\cap\ol{\Gam_w}.
\end{equation}

Then the linear boundary value problem
\begin{align}
\label{lbvp-1}
&\begin{cases}
L_1(v,W)=\Div {\bf F}\\
L_2(v,W)=f
\end{cases} \;\;\tx{in}\;\;\N\\
\label{lbvp-2}
&v=0\;\;\tx{on}\;\;\Gam_0,\quad
\der_{{\bf n}_w}v=0\;\;\tx{on}\;\;\Gam_w,\quad
B_{\q}(\Phi_0,\nabla\psi_0)\cdot \nabla v=g\;\;\tx{on}\;\;\Gam_L\\
\label{lbvp-3}
&W=\begin{cases}
W_{en}&\tx{on}\;\;\Gam_0\\
W_{ex}&\tx{on}\;\;\Gam_{L}
\end{cases},\quad \der_{{\bf n}_w}W=0\;\;\tx{on}\;\;\Gam_w
\end{align}
has a unique solution $(v,W)\in [C^{2,\alp}_{(-1-\alp,\Gam_0\cup\Gam_L)}(\N)]^2$ satisfying the estimate
\begin{equation}
\label{3-b9}
\begin{split}
&\|v\|_{2,\alp,\N}^{(-1-\alp,\Gam_0\cup\Gam_L)}
+\|W\|_{2,\alp,\N}^{(-1-\alp,\Gam_0\cup\Gam_L)}\le C^{\sharp},
\end{split}
\end{equation}
where the constant $C^{\sharp}$ depends only on the data and $\alp$.
\end{proposition}

\begin{corollary}
\label{corollary-est}
If the condition \eqref{normalization} is dropped in Proposition \ref{proposition-3-1}, then the linear boundary value problem of \eqref{lbvp-1}--\eqref{lbvp-3} has a unique solution $(v,W)\in [C^{2,\alp}_{(-1-\alp,\Gam_0\cup\Gam_L)}(\N)]^2$ with satisfying the estimate
\begin{equation}
\label{3-b92}
\begin{split}
&\|v\|_{2,\alp,\N}^{(-1-\alp,\Gam_0\cup\Gam_L)}
+\|W\|_{2,\alp,\N}^{(-1-\alp,\Gam_0\cup\Gam_L)}\\
&\le C^{\sharp}(\|{\bf F}\|_{1,\alp,\N}^{(-\alp,\Gam_0\cup\Gam_L)}+\| f\|_{\alp,\N}+\|g\|_{\alp,\Gam_L}
+\|W_{en}\|_{2,\alp,\Gam_0}+\|W_{ex}\|_{2,\alp,\Gam_L})
\end{split}
\end{equation}
for the constant $C^{\sharp}$ same as in Proposition \ref{proposition-3-1}.
\begin{proof}
Set $m_0:=\|{\bf F}\|_{1,\alp,\N}^{(-\alp,\Gam_0\cup\Gam_L)}+\| f\|_{\alp,\N}+\|g\|_{\alp,\Gam_L}
+\|W_{en}\|_{2,\alp,\Gam_0}+\|W_{ex}\|_{2,\alp,\Gam_L}.$ Then \eqref{3-b92} follows by applying Proposition \ref{proposition-3-1} to $(\frac{v}{m_0},\frac{W}{m_0})$.
\end{proof}
\end{corollary}
In order to prove Proposition \ref{proposition-3-1}, first of all, we achieve uniform $H^1$ estimate of $(v,W)$ under the assumption of \eqref{normalization}, then we combine the uniform $H^1$ estimate with $L^2$ integral growth estimate of $(Dv,DW)$ to get uniform $C^{\alp}$ estimate of $(v,W)$ in $\N$. $H^1$ and $C^{\alp}$ estimate of $(v,W)$ are obtained through estimates of weak solutions to the linear elliptic system \eqref{lbvp-1}. For uniform estimate of $(v,W)$ in $C^{1,\alp}$-norms or higher, we apply a priori $C^{\alp}$ estimate of weak solutions and use individual elliptic equations of the system \eqref{lbvp-1} in the correct order.

\subsection{$H^1$ estimates}
\label{section3-2}
Define
\begin{equation*}
\mH_1:=\{\xi\in H^1(\N):\xi=0\;\;\tx{on}\;\;\Gam_0\},\;\;
\mH_2:=\{\eta\in H^1(\N):\eta=0\;\;\tx{on}\;\;\Gam_0\cup\Gam_L\}.
\end{equation*}
Suppose that $(v,W)\in[C^1(\ol{\N})\cap C^2(\N)]^2$ solve \eqref{lbvp-1}--\eqref{lbvp-3} then, for any $(\xi,\eta)\in\mH_1\times \mH_2$, we have
\begin{equation}
\label{3-b7}
\int_{\N}L_1(v,W)\xi+L_2(v,W)\eta\;d\rx
=\int_{\N}(\Div{{\bf F}})\xi+{f}\eta\;d\rx.
\end{equation}
For the fixed functions $W_{en}$ and $W_{ex}$ in \eqref{lbvp-3}, we define a function $W_{bd}$ in $\N$ by
\begin{equation}
\label{bd-functions}
%v_{bd}(\rx)=v_{en}(\rx'),\quad
W_{bd}(\rx)=(1-\frac{x_n}{L})W_{en}(\rx')+\frac{x_n}{L}W_{ex}(\rx')
\end{equation}
where $\rx'\in\Lambda$ for $\rx=(\rx',x_n)\in\N$. Then, for any $(\xi,\eta)\in\mH_1\times \mH_2$, $(\til v, \til W):=(v,W)-(0, W_{bd})\in \mH_1\times \mH_2$ satisfy
\begin{equation}
\label{3-c1}
\mcl{L}[(\til v, \til W),(\xi,\eta)]=\langle ({\bf F},f,g,W_{bd}),(\xi,\eta)\rangle
\end{equation}
where
\begin{equation}
\label{3-c2}
\begin{split}
&\mcl{L}[(\til v, \til W),(\xi,\eta)]:=\mcl{L}_1[(\til v, \til W),\xi]+\mcl{L}_2[(\til v, \til W),\eta],\\
&\mcl{L}_1[(\til v, \til W),\xi]:=\int_{\N}a_{ii}\der_i\til v\der_i\xi+\til W\der_z{\bf A}(\Psi_0,D\vphi_0)\cdot D\xi\;d\rx,\\
&\mcl{L}_2[(\til v, \til W),\eta]:=\int_{\N}D\til W\cdot D\eta
+(\til W\der_zB(\Psi_0,D\vphi_0)+\der_{\bf q}B(\Psi_0,D\vphi_0)\cdot D\til v)\eta\;d\rx,
\end{split}
\end{equation}
and
\begin{equation}
\label{3-c3}
\begin{split}
&\langle ({\bf F},f,g,W_{bd}),(\xi,\eta)\rangle
:=\underset{(=:I^{(1)})}{\underbrace{\langle({\bf F},f,g),(\xi,\eta)\rangle_1}}
+\underset{(=:I^{(2)})}{\underbrace{\langle W_{bd},(\xi,\eta)\rangle_2}}
\\
&I^{(1)}(\xi,\eta)=\int_{\N}({\bf F}\cdot D\xi-f\eta)\;d\rx-\int_{\der\N\setminus \Gam_0}({\bf F}\cdot{\bf n}_{out})\xi\;dS+\int_{\Gam_L}\frac{a_{nn}p'(B(\Phi_0,\nabla\vphi_0))g\xi}{J_0}\;dS\\
&I^{(2)}(\xi,\eta)=-\int_{\N}W_{bd}\der_z{\bf A}(\Psi_0,D\vphi_0)\cdot D\xi
+(W_{bd}\der_zB(\Psi_0,D\vphi_0)-\Delta W_{bd})\eta\;d\rx
\end{split}
\end{equation}
for unit outward normal vector ${\bf n}_{out}$ on $\der\N\setminus \Gam_0$. This can be directly checked by applying the integration by parts to \eqref{3-b7} and using \eqref{conormal-2}. In \eqref{3-c2} and hereafter, the Einstein summation convention is used. We note that the compatibility condition \eqref{compatibility-linear} implies $\der_{{\bf n}_w}W_{bd}=0$ on $\Gam_w$.
\begin{definition}
\label{def-weaksol}
We call $(v,W)$ a weak solution to \eqref{lbvp-1}--\eqref{lbvp-3} if $(\til v,\til W)=(v,W)-(0,W_{bd})\in \mcl{H}_1\times \mcl{H}_2$ satisfy \eqref{3-c1} for any $(\xi,\eta)\in \mcl{H}_1\times \mcl{H}_2$.
\end{definition}

\begin{lemma}
\label{lemma-3-2}
Under the same assumption as Proposition \ref{proposition-3-1}, the linear boundary value problem \eqref{lbvp-1}--\eqref{lbvp-3} has a unique weak solution $(v,W)\in [H^1(\N)]^2$, and $(\til v,\til W):=(v,W)-(0,W_{bd})$ satisfy
\begin{equation}
\label{3-c4}
\|\til v\|_{H^1(\N)}+\|\til W\|_{H^1(\N)}\le
C_H
\end{equation}
for a constant $C_H$ depending only on the data.
\begin{proof}
Note that $\mcl{H}:=\mcl{H}_1\times \mcl{H}_2$ is a Hilbert space equipped with an inner product $\langle \cdot,\cdot\rangle_{\mH}$ defined by
 \[
 \langle(\xi_1,\eta_1), (\xi_2,\eta_2)\rangle_{\mH}=\int_{\N}(D\xi_1\cdot D\xi_2+D\eta_1\cdot D\eta_2)\;d\rx.
 \]
 By Poincar\'{e} inequality, there exists a constant $\beta_0>0$ satisfying
\begin{equation}
\label{3-c5}
\beta_0(\|\xi\|^2_{H^1(\N)}+\|\eta\|^2_{H^1(\N)})\le \langle(\xi,\eta), (\xi,\eta)\rangle_{\mH}
\end{equation}
for any $(\xi,\eta)\in\mcl{H}$,
and such a constant $\beta_0$ can be chosen depending only on $n$ and $\N$.

Thanks to Lemma \ref{lemma-3-1}(c), we have $\eta\der_z{\bf A}(\Psi_0,D\vphi_0)\cdot D\xi+\eta\der_{\bf q}B(\Psi_0,D\vphi_0)\cdot D\xi=0$ for any $(\xi,\eta)\in \mcl{H}$. Then the definition of $\mcl{L}[\cdot,\cdot]$ in \eqref{3-c2} yields
\begin{equation}
\label{3-c6}
\mcl{L}[(\xi,\eta),(\xi,\eta)]=
\int_{\N}a_{ii}\der_i\xi\der_i\xi+|D\eta|^2
+\der_zB(\Phi_0,D\vphi_0)\eta^2\;d\rx.
\end{equation}
Combining \eqref{3-a2}, Lemma \ref{lemma-3-1}(a) and \eqref{background-sol-prop} together, we also get
\begin{equation}
\label{3-c7}
\mcl{L}[(\xi,\eta),(\xi,\eta)]\ge \lambda_0\int_{\N}|D\xi|^2+|D\eta|^2\;d\rx
\end{equation}
for $\lambda_0=\min(\lambda, 1)$ where $\lambda$ is the one appeared in \eqref{3-b1}.
On the right-hand side of \eqref{3-c1}, applying H\"{o}lder inequality, trace inequality and Poincar\'{e} inequality to $\langle ({\bf F},f,g,W_{bd}),(\xi,\eta)\rangle$, one can find a constant $C_{nh}>0$ depending only on $n$ and $\N$ to satisfy the estimates
\begin{equation}
\label{3-c8}
\begin{split}
&\big|\langle ({\bf F},f,g,W_{bd}),(\xi,\eta)\rangle\big|\le C_{nh} \left(\int_N|D\xi|^2+|D\eta|^2\;d\rx\right)^{1/2}.\\
\end{split}
\end{equation}
This can be directly checked from \eqref{3-c3}.
We note that $C_{nh}$ is independent of ${\bf F},f,g$ and $W_{bd}$ because of \eqref{normalization} and \eqref{bd-functions}.
Then, the Lax-Milgram theorem implies that there exists unique $(\til v,\til W)\in\mcl{H}$ satisfying \eqref{3-c1} for all $(\xi,\eta)\in \mcl{H}$.
Finally, substituting $(\xi,\eta)=(\til v,\til W)$ into \eqref{3-c1}, and applying \eqref{3-c7} and \eqref{3-c8} give
\begin{equation}
\label{h1-est}
\lambda_0\int_{\N}|D\til v|^2+|D\til W|^2\;d\rx\le C_{nh}\left(\int_{\N}|D\til v|^2+|D\til W|^2\;d\rx\right)^{1/2}.
\end{equation}
\eqref{3-c4} follows from applying H\"{o}lder inequality and Poincar\'{e} inequality to \eqref{h1-est}.
\end{proof}

\end{lemma}

\subsection{$C^{\alp}$ estimate}
\label{section3-3}

\begin{lemma}
\label{lemma-3-3}
Under the same assumption as Proposition \ref{proposition-3-1}, let $(\til v, \til W)$ be as in Lemma \ref{lemma-3-2}.
Then there exists $\bar{\alp}\in(0,1)$ depending only on the data so that, for any $\alp\in(0,\bar{\alp}]$, $(\til v, \til W)$ satisfy
\begin{equation}
\label{3-d1}
\|\til v\|_{\alp,\N}+\|\til W\|_{\alp,\N}\le
C_A
\end{equation}
where the constant $C_A$ is chosen depending only on the data and $\alp$.
\begin{proof}
Since $\N$ is a cylindrical domain with the cross-section $\Lambda$, there is a constant $R_0>0$ depending only on $n$ and $\Lambda$ such that for any $\rx_*\in \N$ and $r\in(0,R_0]$, the inequality
\begin{equation*}
\frac{1}{\zeta_0}\le \frac{\vol(B_r(\rx_*)\cap\mcl{N})}{\vol(B_r(\rx_*))}\le \zeta_0
\end{equation*}
holds for a uniform constant $\zeta_0>0$. So if there are constants $R_1\le R_0$ and $\kappa_0>0$ satisfying
\begin{equation}
\label{3-d2}
\int_{B_{r}(\rx_*)\cap\mcl{N}}|D\til v|^2+|D\til W|^2\;d\rx\le {\kappa_0}^2r^{n-2+2\alp}
\end{equation}
for all $\rx_*\in\N$ and $r\in(0,R_1]$, then it follows from \cite[Theorem 3.1]{Ha-L} that
\begin{equation*}
|\til v|_{\alp,\N}+|\til W|_{\alp,\N}\le C\left(\kappa_0+\|\til v\|_{L^2(\N)}+\|\til W\|_{L^2(\N)}\right)
\end{equation*}
for a constant $C$ depending only on $n,\alp,\Lambda$ and $L$.  Combining this with Lemma \ref{lemma-3-2} gives
\begin{equation*}
%\label{3-e4}
|\til v|_{\alp,\N}+|\til W|_{\alp,\N}\le C(\kappa_0+C_H),
\end{equation*}
and this estimate proves \eqref{3-d1}.
To obtain \eqref{3-d2}, we need to consider three cases: (i) $B_r(\rx)\subset\mcl{N}$, (ii) $\rx\in\der\N\setminus \mcl{C}$ for $\mcl{C}:=(\ol{\Gam_0}\cup\ol{\Gam_L})\cap \ol{\Gam_w},$ and $B_r(\rx)\cap  \mcl{C}=\emptyset$, (iii) $\rx\in \mcl{C}$. More general cases can be treated via these three cases. Also, cases (i) and (ii) are easier to handle than case (iii). So we treat only case (iii) here.

Fix $\rx_0\in \mcl{C}$. Without loss of generality, we assume $\rx_0\in \mcl{C}\cap \ol{\Gam_L}$. For a constant $r$ with $0<r\le \frac{1}{10}\min(L, \tx{diam}\; \Lambda, 1)$, set $\mcl{D}_{r}(\rx_0):=\N\cap B_{r}(\rx_0)$. Let $w_1$ and $w_2$ be weak solutions of the following problems:
\begin{equation*}
\begin{split}
&\begin{cases}
\der_i(\sum_{j=1}^na_{ij}(\rx_0)\der_jw_1)=0&\tx{in}\;\;\mcl{D}_{r}(\rx_0),\\
(\sum_{j=1}^na_{ij}(\rx_0)\der_jw_1)\cdot{\bf n}_{out}=0&\tx{on}\;\;\der \mcl{D}_{r}(\rx_0)\cap(\Gam_L\cup\Gam_w),\\
w_1=\til v&\tx{on}\;\;\der\mcl{D}_{r}(\rx_0)\cap\N,
\end{cases}
\end{split}
\end{equation*}
and
\begin{equation*}
\begin{split}
&\begin{cases}
\Delta w_2=0&\tx{in}\;\;\mcl{D}_{r}(\rx_0),\\
\der_{{\bf n}_{out}}w_2=0&\tx{on}\;\;\der \mcl{D}_{r}(\rx_0)\cap(\Gam_L\cup\Gam_w),\\
w_2=\til W&\tx{on}\;\;\der\mcl{D}_{r}(\rx_0)\cap\N,
\end{cases}
\end{split}
\end{equation*}
for unit outward normal vector ${\bf n}_{out}$ on $\der\mcl{D}_r\cap(\Gam_L\cup\Gam_w)$.
In other words, $(w_1,w_2)$ satisfy
\begin{equation}
\label{3-d3}
\int_{\mcl{D}_r(\rx_0)}a_{ij}(\rx_0)\der_jw_1\der_iz_1+Dw_2\cdot Dz_2\;d\rx=0
\end{equation}
for any $z_1,z_2\in H^1(\mcl{D}_r(x_0))\cap\{z=0\;\tx{on}\;\der\mcl{D}_r(x_0)\cap \N\}$. We extend $z_1$ and $z_2$ by setting $z_1=z_2=0$ in $\N\setminus \mcl{D}_r(x_0)$, and substitute $(\xi,\eta)=(z_1,z_2)$ into \eqref{3-c1}. Then subtracting \eqref{3-d3} from \eqref{3-c1} with $(\xi,\eta)=(z_1,z_2)$ gives
\begin{equation}
\label{3-d4}
\begin{split}
&\int_{\mcl{D}_r(\rx_0)}a_{ij}(\rx_0)\der_i(\til v-w_1)\der_jz_1+D(\til W-w_2)\cdot Dz_2\;d\rx\\
&\quad\quad \quad=\langle ({\bf F},f,g,W_{bd}),(z_1,z_2)\rangle -\int_{\mcl{D}_r(x_0)}\mcl{T}\;d\rx,
\end{split}
\end{equation}
where $\mcl{T}$ is defined by $\mcl{T}=\mcl{T}_1+\mcl{T}_2+\mcl{T}_3$ with
\begin{equation*}
\begin{split}
&\mcl{T}_1=(a_{ij}(\rx_0)-a_{ij}(\rx))\der_j\til v\der_iz_1,\quad \mcl{T}_2=\til W\der_z{\bf A}(\Psi_0,D\vphi_0)\cdot Dz_1,\\
&\mcl{T}_3=(\til W\der_zB(\Psi_0,D\vphi_0)+\der_{\bf q}B(\Psi_0,D\vphi_0)\cdot D\til v)z_2.
\end{split}
\end{equation*}
Set $(V_1,V_2):=(\til v,\til W)-(w_1,w_2)$. We substitute $(z_1,z_2)=(V_1,V_2)$ into \eqref{3-d4} and use \eqref{3-b1} to get
\begin{equation}
\label{3-d5}
\int_{\mcl{D}_r(\rx_0)}a_{ij}(\rx_0)\der_iV_1\der_jV_1+DV_2\cdot DV_2\;d\rx
\ge \lambda_0\int_{\mcl{D}_r(x_0)}|DV_1|^2+|DV_2|^2\;d\rx
\end{equation}
for $\lambda_0$ same as in \eqref{3-c7}. It remains to estimate the right-hand side of \eqref{3-d4}. Using Lemma \ref{lemma-3-1}(b) and H\"{o}lder inequality gives
\begin{equation}
\label{3-d6}
\begin{split}
\int_{\mcl{D}_r(\rx_0)}|\mcl{T}_1|\;d\rx
&\le\|a_{ij}\|_{C^1(\ol{N})}
\left(r^2\int_{\mcl{D}_r(\rx_0)}|D\til v|^2\;d\rx\right)^{1/2}\left(\int_{\mcl{D}_r(\rx_0)}|DV_1|^2 \;d\rx\right)^{1/2} \\
&\le \|a_{ij}\|_{C^1(\ol{N})}
r\|\til v\|_{H^1(\N)} \left(\int_{\mcl{D}_r(\rx_0)}|DV_1|^2 \;d\rx\right)^{1/2} .
\end{split}
\end{equation}
Prior to estimate of $\int_{\mcl{D}_r(\rx_0)}|\mcl{T}_2|\;d\rx$, it is necessary to take a closer look at $\int_{\mcl{D}_r(\rx_0)}|\til{W}|^2\;d\rx$. For any $p>2$, H\"{o}lder inequality gives
\begin{equation*}
\int_{\mcl{D}_r(\rx_0)}|\til W|^2\;d\rx\le C(n,p)r^{n(1-\frac 2p)}
\left(\int_{\mcl{D}_r(\rx_0)}|\til W|^p\right)^{\frac 2p}
\end{equation*}
where $C(n,p)$ is a constant depending only $n, \Lambda$ and $p$.
We choose $p=\frac{2n}{n-2}$ for $n>2$, and apply Sobolev inequality to get
\begin{equation}
\label{l2-est}
\int_{\mcl{D}_r(\rx_0)}|\til W|^2\;d\rx\le Cr^2\|\til W\|^2_{H^1(\N)}.
\end{equation}
from which it follows that
\begin{equation}
\label{3-d7}
\begin{split}
&\int_{\mcl{D}_r(\rx_0)}|\mcl{T}_2|\;d\rx\le C\mu_1
r\|\til W\|_{H^1(\N)}
\left(\int_{\mcl{D}_r(\rx_0)}|DV_1|^2\;d\rx\right)^{1/2}%\\
%&\int_{\mcl{D}_r(x_0)}|\mcl{T}_3|\;d\rx\le 4\mu_1
%\left(r^2\int_{\mcl{D}_r(x_0)}|D\til v|^2+|D\til W|^2\;d\rx\right)^{1/2}
%\left(\int_{\mcl{D}_r(x_0)}|DV_2|^2\;d\rx\right)^{1/2}
\end{split}
\end{equation}
for $\mu_1=\|\der_z{\bf A}(\Psi_0,D\vphi_0)\|_{L^{\infty}(\N)}+\|\der_zB(\Psi_0,D\vphi_0)\|_{L^{\infty}(\N)}+\|\der_{\bf q}B(\Psi_0,D\vphi_0)\|_{L^{\infty}(\N)}$.
The following inequality is obtained by H\"{o}lder inequality and Poincar\'{e} inequality:
\begin{equation}
\label{3-d7n}
\begin{split}
%&\int_{\mcl{D}_r(\rx_0)}|\mcl{T}_2|\;d\rx\le \mu_1
%\left(r^2\int_{\mcl{D}_r(\rx_0)}|D\til W|^2\;d\rx\right)^{1/2}
%\left(\int_{\mcl{D}_r(\rx_0)}|DV_1|^2\;d\rx\right)^{1/2}\\
&\int_{\mcl{D}_r(x_0)}|\mcl{T}_3|\;d\rx\le C\mu_1r\left(\|\til v\|_{H^1(\N)}+\|\til W\|_{H^1{\N}}\right)
\left(\int_{\mcl{D}_r(x_0)}|DV_2|^2\;d\rx\right)^{1/2}.
\end{split}
\end{equation}
From \eqref{3-c3}, the estimate $|\langle({\bf F}, f,g,W_{bd}, (V_1,V_2))\rangle|\le |I^{(1)}(V_1,V_2)|+|I^{(2)}(V_1,V_2)|$ follows, and
$
|I^{(1)}(V_1,V_2)|\le
|\mcl{J}_1|+|\mcl{J}_2|$ holds
where
\begin{equation*}
\begin{split}
&\mcl{J}_1=\int_{\mcl{D}_r(\rx_0)}({\bf F}\cdot DV_1- f V_2)\;d\rx,\\
&\mcl{J}_2=-\int_{\der\mcl{D}_r(\rx_0)\cap \der\N}({\bf F}\cdot{\bf n}_{out}) V_1\;dS
+\int_{\der\mcl{D}_r(\rx_0)\cap \Gam_L}\frac{a_{nn}p'(B(\Phi_0,\nabla\vphi_0))gV_1}{J_0}\;dS.
\end{split}
\end{equation*}
By using H\"{o}lder inequality, \eqref{normalization} and Lemma \ref{lemma-3-1}(b), we get
\begin{equation}
\label{3-d8}
|\mcl{J}_2|\le
Cr^{\frac{n-1}{2}}
\left(\int_{\der\mcl{D}_r(\rx_0)}|V_1|^2\;d\rx\right)^{1/2},
\end{equation}
where the constant $C$ in \eqref{3-d8} depends only on the data. Unless otherwise specified, we presume that any constant $C$ appearing in various estimates depends only on the data for the rest of the paper even though $C$ may be different in each estimate.
Since $V_1=0$ on $\der\mcl{D}_r(\rx_0)\cap \N$, one can apply the trace inequality and Poincar\'{e} inequality with scaling to obtain
\begin{equation*}
%\label{3-d9}
\left(\int_{\der\mcl{D}_r(\rx_0)}|V_1|^2\;dS\right)^{1/2}\le C\left(r\int_{\mcl{D}_r(\rx_0)}|DV_1|^2\;d\rx\right)^{1/2},
\end{equation*}
and this combined with \eqref{3-d8} gives
\begin{equation}
\label{3-e1}
|\mcl{J}_2|\le Cr^{n/2}\left(\int_{\mcl{D}_r(\rx_0)}|DV_1|^2\;d\rx\right)^{1/2}.
\end{equation}
Also, it is easy to show
\begin{equation}
\label{3-e2}
|\mcl{J}_1|\le C r^{n/2}
\left(\int_{\mcl{D}_r(\rx_0)}|DV_1|^2+|DV_2|^2\;d\rx \right)^{1/2}.
\end{equation}
By \eqref{bd-functions} and \eqref{normalization}, $W_{bd}$ satisfies
\begin{equation}
\label{est-bdfunctions}
\|W_{bd}\|_{2,\alp,\N}\le C,
\end{equation}
and this yields
\begin{equation}
\label{est-I2}
|I^{(2)}(V_1,V_2)|\le Cr^{n/2}\left(\int_{\mcl{D}_r(\rx_0)}|DV_1|^2+|DV_2|^2 \;d\rx \right)^{1/2}.
\end{equation}
It follows from \eqref{3-d5}-\eqref{3-d7n}, \eqref{3-e1}, \eqref{3-e2}, \eqref{est-I2} and Lemma \ref{lemma-3-2} that
\begin{equation}
\label{est-diff}
\int_{\mcl{D}_r(x_0)}|DV_1|^2+|DV_2|^2\;d\rx\le C\left(r^2C_H^2+r^n\right).
\end{equation}
One can easily adjust the proof of \cite[Lemma 4.12]{Ha-L} and use Lemma \ref{lemma-3-1}(a) to show that there exist two constants $\hat\alp\in(0,1)$ and $\hat{C}>0$ depending only on the data such that $(w_1, w_2)$ in \eqref{3-d3} satisfy
\begin{equation}
\label{est-harmonic}
\int_{\mcl{D}_{\varrho}(\rx_0)}|Dw_1|^2+|Dw_2|^2\;d\rx\le \hat C
\left(\frac{\varrho}{r}\right)^{n-2+2\hat{\alp}}
\int_{\mcl{D}_{r}(\rx_0)}|Dw_1|^2+|Dw_2|^2\;d\rx
\end{equation}
holds whenever $0<\varrho\le r$.  Combining \eqref{est-diff} with \eqref{est-harmonic} yields
\begin{equation*}
\int_{\mcl{D}_{\varrho}(\rx_0)}|D\til v|^2+|D\til W|^2\;d\rx\le
\hat C\left(\frac{\varrho}{r}\right)^{n-2+2\hat{\alp}}
\int_{\mcl{D}_{r}(\rx_0)}|D\til v|^2+|D\til W|^2\;d\rx+10C(C_H+1)^2r^2
\end{equation*}
whenever $0<\varrho\le r\le \frac{1}{10}\min(L, \rm{diam} \Lambda, 1)$ where $C$ and $\hat C$ are same as in \eqref{est-diff} and \eqref{est-harmonic}, respectively.
\begin{itemize}
\item[\emph{Case 1.}] For $n>2$, if $n-2+2\hat{\alp}\le 2$, then
\cite[Lemma 3.4]{Ha-L} implies that one can find constants $\bar\alp\in(0,\hat{\alp})$  and $R_0$ with $0<R_0\le \frac{1}{10}\min(L, \tx{diam}\;\Lambda, 1)$ depending only on the data such that that whenever $0<r\le R_0$, $(\til v, \til W)$ satisfy
\begin{equation}
\label{3-e3}
\int_{\mcl{D}_r(\rx_0)}|D\til v|^2+|D\til W|^2\;d\rx\le Cr^{n-2+2\bar\alp}\;\;\tx{for any}\;\;\rx_0\in \mcl{C}\cap\ol{\Gam_L}.
\end{equation}
\item[\emph{Case 2.}] For $n>2$, if $n-2+2\hat{\alp}> 2$, then
\cite[Lemma 3.4]{Ha-L} implies that there exists a constant $R_1$ with $0<R_1\le \frac{1}{10}\min(L, \tx{diam}\;\Lambda, 1)$ depending only on the data such that that whenever $0<r\le R_1$, $(\til v, \til W)$ satisfy
\begin{equation}
\label{3-e3n}
\int_{\mcl{D}_r(\rx_0)}|D\til v|^2+|D\til W|^2\;d\rx\le Cr^2\;\;\tx{for any}\;\;\rx_0\in \mcl{C}\cap\ol{\Gam_L}.
\end{equation}
By adjusting the proof of \cite[Lemma 3.3, Theorem 3.8]{Ha-L} with using \eqref{3-e3n}, one can find a constant $\bar\alp\in(0,\hat{\alp})$ and $R_2\in (0,R_1)$ depending only on the data such that
\begin{equation*}
\int_{\mcl{D}_r(\rx_0)}|\til W|^2\;d\rx\le Cr^{n-2+2\bar\alp}
\end{equation*}
holds whenever $0<r\le R_2$, and this gives
\begin{equation*}
\int_{\mcl{D}_r(\rx_0)}|D\til v|^2+|D\til W|^2\;d\rx\le Cr^{n-2+2\bar\alp}\;\;\tx{for any}\;\;\rx_0\in \mcl{C}\cap\ol{\Gam_L},\quad 0<r\le R_2.
\end{equation*}
\end{itemize}
The case of $n=2$ can be handled similarly.
Also, one can easily show that \eqref{3-e3} holds for all $x_0\in \ol{\N}$ by repeating the argument above for any given $\rx_0\in\ol{\N}$.
\end{proof}
\end{lemma}

\subsection{$C^{1,\alp}$ estimates}
\label{section4-2}
Under the assumption of Proposition \ref{proposition-3-1}, let $(v,W)$ be a weak solution of \eqref{lbvp-1}--\eqref{lbvp-3} in the sense of Definition \ref{def-weaksol}. Then Lemma \ref{lemma-3-3} combined with \eqref{est-bdfunctions} provides $C^{\alp}$ estimate of $(v,W)=(\til v,\til W)+(0,W_{bd})$ for $\alp\in(0,\bar\alp]$. In order to get uniform $C^{1,\alp}$ estimates of $(v,W)$, we regard the elliptic system \eqref{lbvp-1} as two separate elliptic equations. This is possible because the principal parts of \eqref{lbvp-1} are not coupled. The H\"{o}lder gradient estimate for first equation in \eqref{lbvp-1} gives uniform $C^{1,\alp}$ estimate of $v$. Then $C^{1,\alp}$ estimate of $W$ is obtained by applying Lemma \ref{lemma-3-3} and uniform $C^{1,\alp}$ estimate of $v$.
\begin{lemma}
\label{lemma-3-6}
Under the same assumption as Proposition \ref{proposition-3-1}, let $(\til v, \til W)$ be as in Lemma \ref{lemma-3-2}.
Then, for any $\alp\in(0,1)$, there exists a constant $C_B$ depending only on the data and $\alp$ so that $(v, W)$ satisfy
\begin{equation}
\label{3-g3}
\|v\|_{1,\alp,\N}+\|W\|_{1,\alp,\N}\le
C_B.
\end{equation}
\begin{proof}
First, we prove $C^{1,\alp}$ regularity of $(v,W)$ for $\alp\in(0,\bar\alp]$ where $\bar\alp$ is from Lemma \ref{lemma-3-3}, then a bootstrap argument will give $C^{1,\alp}$ regularity of $(v,W)$ for any $\alp\in(0,1)$. 

Fix $\alp\in(0,\bar{\alp}]$. By \eqref{est-bdfunctions}, if we show that $(\til v,\til W)=(v,W)-(0,W_{bd})$ satisfy
\begin{equation}
\label{3-g9}
\|\til v\|_{1,\alp,\N}+\|\til W\|_{1,\alp,\N}\le
C
\end{equation}
for a constant $C$ depending only on the data and $\alp$, then \eqref{3-g3} easily follows.

{Step 1.}
By choosing $\eta=0$ in \eqref{3-c1}, $\til v$ can be regarded as a weak solution to
\begin{equation}
\label{3-e5}
\begin{cases}
\der_i(\sum_{j=1}^n a_{ij}\der_j\til v)=\Div({\bf F}-\til W\der_z{\bf A}(\Phi_0,D\vphi_0))=:\Div {\bf F}^*&\tx{in}\;\;\N,\\
\til v=0&\tx{on\;\;}\Gam_0,\\
(\sum_{j=1}^n a_{ij}\der_j\til v)\cdot {\bf n}_w=0&\tx{on}\;\;\Gam_w,\\
(\sum_{j=1}^n a_{ij}\der_j\til v)\cdot {\bf n}_L=\frac{a_{nn}p'(B(\Phi_0,\nabla\vphi_0))g}{J_0}=:\mfh&\tx{on}\;\;\Gam_L.
\end{cases}
\end{equation}
By Lemma \ref{lemma-3-3}, ${\bf F}^*$ satisfies the estimate
\begin{equation}
\label{3-e6}
\|{\bf F}^*\|_{\alp,\N}\le C.
\end{equation}
 Since ${\mfh}\in C^{\alp}(\ol{\Gam_L})$, one can find a function $v_{\mfh}\in C^{1,\alp}(\ol{\N\cap\{x_n>\frac L4\}})$ satisfying
 \begin{equation}
 \label{4-a2}
 (\sum_{j=1}^n a_{ij}\der_j v_{\mfh})\cdot {\bf n}_L=\mfh\quad\tx{on}\quad\Gam_L\quad\tx{and}\quad
 \|v_{\mfh}\|_{1,\alp,\N\cap\{x_n>\frac L4\}}\le C.
 \end{equation}
One can refer to Step 2 in the proof of Lemma 3.1 of \cite{Ch-F3} for more details on the existence of such $v_{\mfh}$. It is easy to see that $V_{\mfh}:=\til v-v_{\mfh}$ is a weak solution to
\begin{equation}
\label{3-e7}
\begin{cases}
\der_i(\sum_{j=1}^n a_{ij}\der_jV_{\mfh})=\Div({\bf F}^*-\sum_{j=1}^na_{ij}\der_j v_{\mfh})=:\Div \hat{\bf H}&\tx{in}\;\;\N\cap\{x_n>L/4\},\\
(\sum_{j=1}^n a_{ij}\der_jV_{\mfh})\cdot{\bf n}_w=-(\sum_{j=1}^n a_{ij}\der_j v_{\mfh})\cdot{\bf n}_w=:\hat g_1&\tx{on}\;\;\Gam_w\cap\{x_n>L/4\},\\
V_{\mfh}=\til v-v_{\mfh}=:\hat g_2&\tx{on}\;\;\ol{\N}\cap\{x_n=L/4\},\\
(\sum_{j=1}^n a_{ij}\der_j V_{\mfh})\cdot {\bf n}_L=0&\tx{on}\;\;\Gam_L.
\end{cases}
\end{equation}
For $l\in\R^+$, set $Q^+_{l}:=\Lambda\times (l,2L-l)$. For $(\rx',x_n)\in Q^+_{\frac L4}$, define
\begin{equation}
\label{3-e9}
a_{ij}(\rx',x_n)=a_{ij}(\rx',2L-x_n),\;\;
\hat g_k(\rx',x_n)=\hat g_k(\rx', 2L-x_n)
\end{equation}
for $i,j=1,\cdots,n$, $k=1,2$ and for all $\rx'=(x_1,\cdots,x_{n-1})\in\ol{\Lambda}$.    Also, we define ${\bf H}=(H_1,\cdots,H_n)$ in $Q^+_{\frac L4}$ by
\begin{equation}
\label{3-f1}
\begin{split}
&H_j(\rx',x_n)=\begin{cases}
\hat H_j(\rx',x_n)&\tx{for}\;\;x_n\le L\\
\hat H_j(\rx',2L-x_n)&\tx{for}\;\;x_n>L
\end{cases}\;\;\tx{for}\;\;j\neq n,
\end{split}
\end{equation}
and
\begin{equation}
\begin{split}
&H_n(\rx',x_n)=
\begin{cases}
\hat H_n(\rx',x_n)-\hat H_n(\rx',L)&\tx{for}\;\;x_n\le L,\\
-(\hat H_n(\rx',2L-x_n)-\hat H_n(\rx',L))&\tx{for}\;\;x_n>L.
\end{cases}
\end{split}
\end{equation}
According to \eqref{3-e9} and \eqref{3-f1}, $a_{ij}, \hat g_k$ and each $\hat H_j(j\neq n)$ are extended from $\N\cap\{x_n>\frac L4\}$ to $Q^+_{\frac L4}$ by even reflection about $x_n=L$, and $H_n$ is defined as the odd extension of $\hat H_n(\rx',x_n)-\hat H_n(\rx',L)$ about $x_n=L$.
Hence $a_{ij}\in C^{0,1}(\ol{Q^+_{\frac L4}})$, ${\bf H}\in C^{\alp}(\ol{Q^+_{\frac L4}})$ and $\hat g_1\in C^{\alp}(\der\Lambda\times\{|x_n-L|<3L/4\})$. For such extended functions $a_{ij}, \hat g_1, \hat g_2$ and ${\bf H}$, the boundary value problem
\begin{equation}
\label{3-e8}
\begin{cases}
\der_i(\sum_{j=1}^n a_{ij}\der_j U)=\Div{\bf H},&\tx{in}\;\;Q^+_{\frac L4}\\
(\sum_{j=1}^na_{ij}\der_j U)\cdot {\bf n}_w=\hat g_1&\tx{on}\;\;\der\Lambda\times\{|x_n-L|<3L/4\},\\
U=\hat g_2&\tx{on}\;\; \ol{\Lambda}\times\{|x_n-L|=3L/4\}
\end{cases}
\end{equation}
has a unique weak solution $U\in H^1(Q^+_{\frac L4})$. Furthermore, the assumption \eqref{normalization} together with the proof of \cite[Lemma 3.1]{Ch-F3} shows that
\begin{equation}
\label{4-a1}
\|U\|_{1,\alp,Q^+_{\frac L3}}\le C(\|U\|_{L^2(Q^+_{\frac L4})}+1).
\end{equation}
It also follows from \eqref{3-e9} and \eqref{3-f1} that if $U$ is a weak solution of \eqref{3-e8}, then so is $\til U(\rx',x_n)=U(\rx',2L-x_n)$. Then the uniqueness of weak solution yields $U(\rx',x_n)=U(\rx',2L-x_n)$ thus  $\der_n U(\rx',L)=0$ which implies $(\sum_{j=1}^na_{ij}\der_jU)\cdot{\bf n}_L=0$ on $\Gam_L$. From this, $V_{\mfh}=U$ is obtained in $\N\cap\{|x_n-L|<3L/4\}$. Therefore, by Lemma \ref{lemma-3-3}, \eqref{3-e6}, \eqref{4-a2} and \eqref{4-a1}, the estimate
\begin{equation}
\label{3-f2}
\|\til v\|_{1,\alp,\N\cap\{x_n>\frac L3\}}\le C
\end{equation}
holds for a constant $C$ depending only on the data and $\alp$.

{Step 2.} Back to \eqref{3-e5}, by \eqref{3-a2}, Definition \ref{def-background-sol} and \eqref{F-condition}, ${\bf F}^*$ satisfies
 \begin{equation}
 \label{F-entrance}
 {\bf F}^*\cdot \hat{\bf e}_j=0\quad\tx{on}\quad\Gam_0\;\;\tx{for}\;\;j=1,\cdots,n-1.
 \end{equation}
For $l\in\R^+$, set $Q^-_{l}:=\Lambda\times(-l,l)$, and we extend $a_{ij}$ and ${\bf F}^*$ into $Q^-_{3L/4}$ as follows: for $x_n<0$, define
\begin{equation}
\label{4-a5}
a_{ij}(\rx',x_n)=a_{ij}(\rx',-x_n)\;\;\tx{for}\;\;\rx' \in\ol{\Lambda},\quad i,j=1,\cdots,n.
\end{equation}
Define $\til{\bf F}=(\til F_1,\cdots,\til F_n)$ in $Q^-_{3L/4}$ by
\begin{equation}
\label{4-a6}
\begin{split}
&\til F_j(\rx',x_n)=\begin{cases}
F^*_j(\rx',x_n)&\tx{for}\;\;x_n\ge 0\\
-F^*_j(\rx',-x_n)&\tx{for}\;\;x_n<0
\end{cases}\;\;\tx{for}\;\;j\neq n,\\
&\til F_n(\rx',x_n)=\begin{cases}
F^*_n(x',x_n)&\tx{for}\;\;x_n\ge 0\\
F^*_n(\rx',-x_n)&\tx{for}\;\;x_n<0
\end{cases}.
\end{split}
\end{equation}
It follows from \eqref{normalization} and \eqref{F-entrance} that each $\til F_j$ is in $C^{\alp}(\ol{Q^-_{3L/4}})$ and satisfies $\|\til F_j\|_{\alp,Q_{3L/4}}\le C$. Hence, the boundary value problem
\begin{equation}
\label{4-a7}
\begin{cases}
\der_i(\sum_{j=1}^na_{ij}\der_j \til U)=\Div\; \til{\bf F}&\tx{in}\;\;Q^-_{3L/4}\\
(\sum_{j=1}^na_{ij}\der_j \til U)\cdot{\bf n}_w=0&\tx{on}\;\;\der\Lambda\times(-\frac {3L}4, \frac {3L}4)\\
\til U(\rx',x_n)=(\sgn\; x_n)\til v(x',|x_n|)&\tx{on}\;\;\ol{\Lambda}\times \{|x_n|=\frac{3L}{4}\}
\end{cases}
\end{equation}
has a unique weak solution $\til U\in H^1(Q^-_{3L/4})$. Moreover, $\til U$ satisfies the estimate
\begin{equation}
\label{4-a8}
\|\til U\|_{1,\alp,Q^-_{2L/3}}\le C.
\end{equation}
In fact, \eqref{4-a1} and \eqref{4-a8} can be verified by simplifying the estimates of \cite{lieberman1} to the case of linear elliptic equation(see \cite[Theorems 1.1 and 1.2]{lieberman1}).
By using \eqref{4-a5}, \eqref{4-a6} and the uniqueness of weak
solution to \eqref{4-a7}, one can directly check that $\til U(\rx',-x_n)=-\til U(\rx',x_n)$ in $Q^-_{3L/4}$, and this yields $\til U(x',0)=0$ on $\Gam_0$. Therefore, we get $\til v=\til U$ in $\N\cap\{x_n<3L/4\}$. It follows from \eqref{3-d1} and \eqref{4-a8} that
$\|\til v\|_{1,\alp,\N\cap\{x_n<2L/3\}}\le C$.
Combining this with \eqref{3-f2} yields
\begin{equation}
\label{4-a9}
\|\til v\|_{1,\alp,\N}\le C
\end{equation}
for a constant $C$ depending only on the data.

{Step 3.} Substituting $\xi=0$ into \eqref{3-c1}, $\til W$ can be regarded as a weak solution to
\begin{equation}
\label{3-f3}
\begin{cases}
\Delta \til W=f+(\til W\!+\!W_{bd})\der_zB(\Phi_0,D\vphi_0)\!-\!\Delta W_{bd}+\der_{\bf q}B(\Phi_0,D\vphi_0)\!\cdot\! D\til v=:\mff_*,&\tx{in}\;\;\N,\\
\til W=0,&\tx{on}\;\;\Gam_0\cup\Gam_L,\\
D\til W\cdot{\bf n}_w=0,&\tx{on}\;\;\Gam_w.
\end{cases}
\end{equation}
By using odd extensions of $\mff_*$ and $V$ about $\Gam_0$ and $\Gam_L$, one can argue similarly as Step 2 to show that
\begin{equation}
\label{4-b1}
\|\til W\|_{1,\alp,\N}\le  C
\end{equation}
for a constant $C$ depending only on the data and $\alp$.

We have shown that $(v,W)$ satisfy \eqref{3-g3} for $\alp\in(0,\bar\alp]$ therefore  $v$ and $W$ are $C^{\alp}$ in $\ol{\N}$ for all $\alp\in(0,1)$. Then, one can repeat the argument above to show that \eqref{3-g3} holds for all $\alp\in(0,1)$. 
\end{proof}
\end{lemma}

\subsection{Unique solvability of mixed boundary value problem}
In \S\ref{section3-2}--\S\ref{section4-2}, it is shown that the boundary value problem \eqref{lbvp-1}--\eqref{lbvp-3} has a unique weak solution $(v,W)$ and that the weak solution is $C^{1,\alp}$ up to the boundary of $\N$. In order to prove Proposition \ref{proposition-3-1}, it remains to show that $v$ and $w$ are $C^2$ in $\N$ thus the weak solution is actually a classical solution of \eqref{lbvp-1}--\eqref{lbvp-3}.
\begin{proof}[Proof of Proposition \ref{proposition-3-1}]
By Lemmas \ref{lemma-3-2}--\ref{lemma-3-6}, the linear boundary value problem \eqref{lbvp-1}--\eqref{lbvp-3} has a unique weak solution $(v,W)$ in the sense of Definition \ref{def-weaksol}, and $(v,W)$ satisfy the $C^{1,\alp}$ estimate \eqref{3-g3}. Then $v$ can be regarded as a weak solution of
\begin{equation}
\label{3-h3}
\begin{split}
&\der_i(\sum_{j=1}^n a_{ij}\der_j v)=\Div({\bf F}-W\der_z{\bf A}(\Phi_0,D\vphi_0))=:\Div \bar{\bf F}\quad\tx{in}\quad\N\\
&(\sum_{j=1}^n a_{ij}\der_j v)\cdot {\bf n}_w=0\quad\tx{on}\quad \Gam_w
\end{split}
\end{equation}
for $\bar{\bf F}\in C^{1,\alp}(\N\cup \Gam_w)$. Since every $a_{ij}$ is smooth in $\N$, $v$ is in $C^{2,\alp}(\N)$. In fact, $v$ is $C^{2,\alp}$ up to $\Gam_w$ away from $\ol{\Gam_0}\cup\ol{\Gam_L}$. We give a heuristic argument to verify this. For any given point $\rx_0\in\Gam_w\setminus(\ol{\Gam_0}\cup\ol{\Gam_L})$, $\Gam_w$ can be locally flattened near $\rx_0$, and \eqref{3-h3} can be written as a conormal boundary value problem as follows:
\begin{equation}
\label{3-h4}
\begin{split}
&\der_{y_i}(\sum_{j=1}^n\til a_{ij}\der_{y_j}\til v)=\Div \til{\bf F}\;\;\tx{in}\;\;B_{2R_0}^+({\bf 0}):=B_{2R_0}({\bf 0})\cap\{y_n>0\},\\
&(\sum_{j=1}^n\til a_{ij}\der_{y_j}\til v)\cdot {\til{\bf n}_w}=0\;\;\tx{on}\;\;
\der B_{2R_0}^+({\bf 0})\cap\{y_n=0\}
\end{split}
\end{equation}
for sufficiently small $R_0>0$ and inward unit normal vector $\til{\bf n}_w$ on $\{y_n=0\}$. We note that $\til{\bf n}_w$ is a constant vector. Here, we assume that $\rx_0$ is mapped to ${\bf 0}$ in the flattened domain. Since the boundary of the cross-section $\Lambda$ of $\N$ is assumed to be smooth, one can find a smooth flattening map near $\rx_0$ so that $\til a_{ij}\in C^{3}(\ol{B^+_{R_0}(\bf 0)})$ and $\til{\bf F}\in C^{1,\alp}(\ol{B^+_{R_0}(\bf 0)})$. For each $k=1,\cdots, n-1$, we differentiate \eqref{3-h4} with respect to $y_k$ then $w^{(k)}(y):=\til v_{y_k}(y)$ becomes a weak solution to
\begin{equation*}
\begin{split}
&\der_{y_i}(\sum_{j=1}^n\til a_{ij}\der_{y_j}w^{(k)})=\Div_y(\der_{y_k}\til{\bf F}-\sum_{j=1}^n\der_{y_k}\til a_{ij}\der_{y_j}\til v)=:\Div_y {\bf G^{(\rx_0)}}(y)  \quad\tx{in}\quad B_{R_0}^+({\bf 0})\\
&(\sum_{j=1}^n\til a_{ij}\der_{y_j}w^{(k)})\cdot\til{\bf n}_w=-(\sum_{j=1}^n\til \der_{y_k}a_{ij}\der_{y_j}\til v)\cdot{\til{\bf n}_w}=:g^{(\rx_0)}(y)\quad\tx{on}\quad \der B_{R_0}^+({\bf 0})\cap\{y_n=0\}
\end{split}
\end{equation*}
for ${\bf G}^{(\rx_0)}\in C^{\alp}(\ol{B_{R_0}^+(\bf 0)})$ and $g^{(\rx_0)}\in C^{\alp}(\ol{B_{R_0}^+(\bf 0)})$. Thus $w^{(k)}\in C^{1,\alp}(\ol{B^+_{R_0/2}({\bf 0})})$ for $k=1,\cdots, n-1$, and this implies that $\til v_{y_iy_j}$ is $C^{\alp}$ up to $\der B^+_{R_0}({\bf 0})$ unless $i=j=n$. One can also check that $\til v_{y_ny_n}$ is $C^{\alp}$ up to $\der B^+_{R_0}({\bf 0})$ by solving the first equation in \eqref{3-h4} for $\til v_{y_ny_n}$ because $\til a_{nn}$ is strictly positive in $\ol{B^+_{2R_0}(\bf 0)}$. This proves that $v$ is $C^{2,\alp}$ up to $\Gam_w$ away from $\ol{\Gam_0}\cup\ol{\Gam_L}$. Rigorous proof can be given by using difference quotient(cf.\cite[Section 8.3]{GilbargTrudinger}). Once $v\in C^{2,\alp}$ is shown, we can also show that $W$ is $C^{2,\alp}$ up to $\Gam_w$ away from $\ol{\Gam_0}\cup\ol{\Gam_L}$ using the second equation of \eqref{lbvp-1} and $C^{2,\alp}$ regularity of $v$.  It remains to estimate $C^{2,\alp}_{(-1-\alp,\Gam_0\cup\Gam_L)}$-norms of $v$ and $W$ by a scaling argument. For a fixed point $\ry_0\in\ol{\N}\setminus (\ol{\Gam_0}\cup\ol{\Gam_L})$, let $2d=\text{dist}(\ry_0, \corners)$, and set a scaled function
\begin{equation*}
v^{(\ry_0)}(s):=\frac{1}{d^{1+\alp}}\bigl(v(\ry_0+ds)-v(\ry_0)-dDv(\ry_0)\cdot s\bigr)
\end{equation*}
for $s\in \{s\in B_1(0):\ry_0+ds\in \N\}=:\mcl{B}_1(\ry_0)$ where $B_1(0)$ is the unit ball in $\R^n$ with the center at the origin. By \eqref{3-g3}, we have
\begin{equation}
\label{4-b2}
\|v^{(y_0)}\|_{C^0(\mcl{B}_1(0))}\le \|v\|_{1,\alp,\N}\le C_B.
 \end{equation}
Substituting $v^{(\ry_0)}$ into \eqref{3-h3} gives
 \begin{equation}
 \label{4-b3}
 \begin{cases}
 \der_{s_i}(\sum_{j=1}^na_{ij}(\ry_0+ds)\der_{s_j}v^{(\ry_0)})=\Div_s \bar{\bf F}^{(\ry_0)}\;\;\tx{in}\;\;\mcl{B}_1(0),\\
 (\sum_{j=1}^n a_{ij}(\ry_0+ds)\der_{s_j}v)\cdot{\bf n}_w(\ry_0+ds)=\mathfrak{g}^{(\ry_0)},\;\;\quad\tx{on}\;\;\der\mcl{B}_1(0)\cap\{\ry_0+ds\in\Gam_w\}.\\
 \;(\tx{if $\der\mcl{B}_1(0)\cap\{\ry_0+ds\in\Gam_w\}\neq \emptyset $})
 \end{cases}
 \end{equation}
 for
 \begin{equation}
 \label{3-h1}
 \begin{split}
 &\bar {\bf F}^{(\ry_0)}(s):=\frac{\bar{\bf F}(\ry_0+ds)-\bar{\bf F}(y_0)-\sum_{j=1}^n(a_{ij}(\ry_0+ds)-a_{ij}(\ry_0))\der_j v(\ry_0)}{d^{\alp}},\\
 &\mathfrak{g}^{(\ry_0)}(s):=\frac{(\sum_{j=1}^n a_{ij}(\ry_0+ds)\der_jv(\ry_0))\cdot{\bf n}_w(\ry_0+ds)-(\sum_{j=1}^n a_{ij}(\ry_0)\der_jv(\ry_0))\cdot{\bf n}_w(\ry_0)}{d^{\alp}}.
 \end{split}
 \end{equation}
 By using Lemma \ref{lemma-3-1}(b), \eqref{normalization} and \eqref{3-g3}, one can find a constant $C_F$ depending only on the data and $\alp$ so that, for any $\ry_0\in\ol{\N}\setminus (\corners)$, $\bar{\bf F}^{(\ry_0)}$ and $\mathfrak{g}^{(\ry_0)}$ satisfy
 \begin{equation}
 \label{3-h2}
 \|\bar{\bf F}^{(\ry_0)}\|_{1,\alp,\mcl{B}_1(0)}
 +\|\mathfrak{g}^{(\ry_0)}\|_{1,\alp,\mcl{B}_1(0)}\le C_F.
 \end{equation}
Then, it follows from \cite[Sections 6.3 and 6.7]{GilbargTrudinger}, \eqref{4-b2} and \eqref{3-h2} that $v^{(\ry_0)}$ is in $C^{2,\alp}(\ol{\mcl{B}_{1/2}(0)})$ and satisfies
 \[
  \|v^{(\ry_0)}\|_{2,\alp,\mcl{B}_{1/2}(0)}\le
  C(C_B+C_F)=:C_*
 \]
 for the constant $C_B$ in \eqref{3-g3}. We emphasize that the constant $C_*$ above is independent of $\ry_0\in \ol{\N}\setminus(\corners)$. Therefore, re-scaling back gives
\begin{equation*}
\|v\|_{2,\alp,\N}^{(-1-\alp,\corners)}\le C_{**}
\end{equation*}
for a constant $C_{**}$ depending only on the data and $\alp$.

Similarly, we consider $W$ as a solution of the equation
\begin{equation*}
\Delta W=f+\der_zB(\Phi_0,D\vphi_0)W+\der_{\q}B(\Phi_0,D\vphi_0)\cdot Dv=:\bar{\mff}_*\quad\tx{in}\quad \N.
\end{equation*}
Since $\bar{\mff}_*\in C^{\alp}(\N)$ with satisfying $\|\bar{\mff}_*\|_{\alp,\N}\le C$ by \eqref{3-g3}, we can repeat the scaling argument above to conclude that
\begin{equation*}
\|W\|_{2,\alp,\N}^{(-1-\alp,\corners)}\le \mcl{C}^*.
\end{equation*}
\end{proof}

\section{Proof of Proposition \ref{theorem2}}
\label{section4}
Fix $\alp\in(0,1)$. If $M\sigma\le 2\min\{\delta_1,\delta_2\}$ for $\delta_1, \delta_2$ in \eqref{delta1} and \eqref{delta2} respectively, then for each $(\tpsi,\tPsi)\in\mcl{K}_M$, functions ${\bf F}(\rx, \tPsi, D\tpsi)$, $f(\rx,\tPsi,D\tpsi)$ are well defined in $\N$, and $g(\rx,D\tpsi, \pex,\Psi_{ex})$ is well defined on $\Gam_L$ where ${\bf F}(\rx,z,\q), f(\rx,z,\q)$ and $g(\rx,\q,\pex,\Psi_{ex})$ are defined by \eqref{3-a5}, \eqref{3-a6} and \eqref{3-g8}, respectively.

\begin{lemma}
\label{lemma-4-1}
\label{est-nhfunctions}There exists a positive constant $\delta_3(\le 1)$ depending only on the data so that if $M\sigma\le \delta_3$ in the definition of $\mcl{K}_M$ given in \eqref{3-g1}, then the following properties hold: for any $(\tpsi,\tPsi)\in\mcl{K}_M$,
\begin{align}
\label{nhest-1}
\|{\bf F}(\rx, \tPsi,D\tpsi)\|_{1,\alp,\N}^{(-\alp,\corners)}\le C(M\sigma)^2\\
\label{nhest-2}
\|f(\rx,\tPsi,D\tpsi)\|_{\alp,\N}\le C(M^2\sigma+1)\sigma\\
\label{nhest-3}
\|g(\rx,D\tpsi,\pex,\Psi_{ex})\|_{\alp,\Gam_L}\le C(M^2\sigma+1)\sigma
\end{align}
for a constant $C$ depending only on the data. Furthermore, we have
\begin{equation}
\label{nhest-4}
{\bf F}(\rx,\tPsi,D\tpsi)\cdot\hat{\bf e}_j=0\quad\tx{on}\quad\Gam_0\;\;\tx{for}\;\;j=1,\cdots,n-1.
\end{equation}
\begin{proof}
For $\rho$ defined by \eqref{2-b4}, we can choose a constant $\bar{\delta}>0$ depending only on the data so that if $M\sigma\le \bar\delta$, then there holds
\begin{equation*}
0<\frac 12 \inf_{\N}\rho(\Phi_0,D\vphi_0)\le \rho(\Phi_0+\tPsi,D\vphi_0+D\tpsi)\le 2\sup_{\N}\rho(\Phi_0,D\vphi_0)\quad\tx{in}\quad\N
\end{equation*}
for all $(\tpsi,\tPsi)\in\mcl{K}_M$. Then, by \eqref{2-a5}, the inequality
\begin{equation*}
\mu_0\le p'(\rho(\Phi_0+\tPsi,D\vphi_0+D\tpsi))\le \frac{1}{\mu_0}\quad\tx{in}\quad\N
\end{equation*}
holds for a constant $\mu_0>0$ depending only on the data. Thus \eqref{nhest-1}--\eqref{nhest-3} easily follow from \eqref{data}, \eqref{3-a5}, \eqref{3-a6}, \eqref{3-f9}, \eqref{3-g6}, Lemma \ref{lemma-3-4} and Lemma \ref{lemma-3-5}.

In order to prove \eqref{nhest-4}, we need to take a closer look at the definition of ${\bf F}$ in \eqref{3-a5}. By \eqref{3-a2}, we have $\der_z{\bf A}(\Phi_0+z,D\vphi_0+\q)\parallel D\vphi_0+\q$ for all $(\rx,z,\q)\in \mcl{D}_{3\delta_1}$ where the set $\mcl{D}_{3\delta_1}$ is defined by \eqref{delta1}. Choosing $\delta_3=\min\{\bar{\delta}, \delta_1,\delta_2\}$ so that for any $(\rx,z,\q)\in\N\times\{(z,\q)\in\R\times \R^n:|z|+|\q|\le \delta_3\}(=:\mcl{D}_{\delta_3})$, we have
\begin{equation}
\label{4-a3}
\begin{split}
&\int_0^1 z[\der_zA_i(\Phi_0+\til z, D\vphi_0+\til{\q})]_{(\til z,\til{\q})=(0,{\bf 0})}^{(t z,{\q})}dt=a_1(\rx, z,{\q})D\vphi_0+a_2(\rx, z,{\q}){\q},\\
&\int_0\sum_{j=1}^n{q}_j[\der_{q_j}A_i(\Phi_0,D\vphi_0+\til{\bf q})]_{\til{\bf q}={\bf 0}}^{t{\bf q}}dt=a_{3,i}(\rx, z,{\q}) q_i\quad\tx{for all}\quad i\neq n.
\end{split}
\end{equation}
for $a_1,a_2,a_{3,i}\in C(\ol{\mcl{D}}_{\delta_3},\R)$. For any $(\tpsi,\tPsi)\in\mcl{K}_M$, since $\tpsi=0$ on $\Gam_0$ by the definition of $\mcl{K}_M$ in \eqref{3-g1}, there holds $\der_i\tpsi=0$ on $\Gam_0$ for all $i\neq n$. So if $M\sigma\le \delta_3$, then, by \eqref{3-a5} and \eqref{4-a3}, ${\bf F}(\rx,\tPsi,D\tpsi)$ satisfies \eqref{nhest-4}.
\end{proof}
\end{lemma}
Now we are in position to prove Proposition \ref{theorem2}.
\begin{proof}[Proof of Proposition \ref{theorem2}]
If $M^2\sigma\le 1$, then, it follows from \eqref{data}, Proposition \ref{proposition-3-1} and Lemma \ref{lemma-4-1} that the elliptic system \eqref{3-g2} with boundary conditions of \eqref{3-f6} and \eqref{3-g4} has a unique solution $(\hat{\psi},\hat{\Psi})\in [C^{2,\alp}_{(-1-\alp,\corners)}(\N)]^2$, which satisfies
\begin{equation}
\label{apriori-est}
\|\hat{\psi}\|_{2,\alp,\N}^{(-1-\alp,\corners)}
+\|\hat{\Psi}\|_{2,\alp,\N}^{(-1-\alp,\corners)}
\le 2C^{\flat}\sigma
\end{equation}
for $C^{\flat}$ depending only on the data and $\alp$.

Next, we estimate $\mathfrak{I}(\tpsi_1,\tPsi_1)-\mathfrak{I}(\tpsi_2,\tPsi_2)$ in $C^{2,\alp}_{(-1-\alp,\corners)}(\N)$ for $(\tpsi_k,\tPsi_k)\in\mcl{K}_M (k=1,2)$. For each $k=1,2$, set
\begin{equation*}
\begin{split}
&(\hat\psi_{k},\hat\Psi_{k}):=\mathfrak{I}(\tpsi_k,\tPsi_k)\;\;\tx{in}\;\;N\\
&{\bf F}^{(k)}:={\bf F}(\rx,\tPsi_k, D\tpsi_k),\quad f^{(k)}:=f(\rx,\tPsi_k,D\tpsi_k)\;\;\tx{in}\;\;\N\\
&g^{(k)}:=g(\rx,D\tpsi_k,\pex,\Psi_{ex})\;\;\tx{on}\;\;\Gam_{ex}.
\end{split}
\end{equation*}
Then $(v^*,W^*):=(\hat\psi_1,\hat\Psi_1)-(\hat\psi_2,\hat\Psi_2)$ satisfy
\begin{equation*}
\begin{split}
&\begin{cases}
L_1(v^*,W^*)=\Div({\bf F}^{(1)}-{\bf F}^{(2)})\\
L_2(v^*,W^*)=f^{(1)}-f^{(2)}
\end{cases}\quad\tx{in}\quad\N,\\
&v^*=0\quad\tx{on}\quad\Gam_0,\quad B_{\q}(\Phi_0,D\vphi_0)\cdot Dv^*=g^{(1)}-g^{(2)}\quad\tx{on}\quad \Gam_L,\\
&\der_{{\bf n}_w}v^*=\der_{{\bf n}_w}W^*=0\quad\tx{on}\quad\Gam_w,\\
&W^*=0\quad\tx{on}\quad \Gam_0\cup\Gam_L.
\end{split}
\end{equation*}
By Lemma \ref{lemma-3-4}(a) and Lemma \ref{lemma-3-5}(a), there exists a constant $C$ depending on the data and $\alp$ such that
\begin{equation}
\label{nhest-cont}
\begin{split}
&\|{\bf F}^{(1)}-{\bf F}^{(2)}\|_{1,\alp,\N}^{(-\alp,\corners)}
+\|f^{(1)}-f^{(2)}\|_{\alp,\N}
+\|g^{(1)}-g^{(2)}\|_{\alp,\Gam_L}\\
&\le C(1+M)\sigma(\|\tpsi_1-\tpsi_2\|_{2,\alp,\N}^{(-1-\alp,\corners)}
+\|\tPsi_1-\tPsi_2\|_{2,\alp,\N}^{(-1-\alp,\corners)}).
\end{split}
\end{equation}
By Corollary \ref{corollary-est} and \eqref{nhest-cont}, one can find a constant $C^{\natural}$ depending on the data and $\alp$ so that $(v^*,W^*)$ satisfy
\begin{equation}
\label{est-cont}
\begin{split}
&\|v^*\|_{2,\alp,\N}^{(-1-\alp,\corners)}
+\|W^*\|_{2,\alp,\N}^{(-1-\alp,\corners)}\\
&\le C^{\natural}(1+M)\sigma(\|{\tpsi}_1-{\tpsi}_2\|_{2,\alp,\N}^{(-1-\alp,\corners)}
+\|\tPsi_1-\tPsi_2\|_{2,\alp,\N}^{(-1-\alp,\corners)}).
\end{split}
\end{equation}
To complete the proof, we choose $M$ and $\sigma$ as follows:
\begin{itemize}
\item[(i)] Choose $M$ to be
\begin{equation}
\label{M}
M=\max\{1,4C^{\flat}\}
\end{equation}
for $C^{\flat}$ as in \eqref{apriori-est}. Then $\mathfrak{I}$ maps $\mcl{K}_M$ into itself.
\item[(ii)] Next we choose $\sigma_1$ as
\begin{equation}
\label{sigma}
\sigma_1=\frac {1}{4M} \min\{\delta_1, \delta_2, \delta_3,1,\frac 1M, \frac{1}{2MC^{\natural}}\}
\end{equation}
for $\delta_1, \delta_2, \delta_3$ as in \eqref{delta1}, \eqref{delta2}, Lemma \ref{lemma-4-1}, respectively. Then \eqref{est-cont} implies that $\mathfrak{I}$ is contracting in $[C^{2,\alp}_{(-1-\alp,\corners)}(\N)]^2$.
\end{itemize}
For $M$ as in \eqref{M} and for any $\sigma\in(0,\sigma_1]$, $\mcl{K}_M$ is a nonempty Banach space equipped with the norm of $\|\cdot\|_{2,\alp,\N}^{(-1-\alp,\corners)}$. Then the contraction mapping principle implies that for any $\sigma\in(0,\sigma_1]$, the iteration mapping $\mathfrak{I}$ has a unique fixed point $(\psi_*,\Psi_*)$ in $\mcl{K}_M$. If necessary, we can further reduce $\sigma_1>0$ depending only on the data so that wherever $\sigma\in(0,\sigma_1]$, any fixed point of $\mathfrak{I}$ in $\mcl{K}_M$ satisfies \eqref{subsonicity}.
\end{proof}

\appendix

\section{Nonconstant background density $b$ for one-dimensional flow }
\label{section-nonconstant-b}
\begin{proposition}
\label{proposition-app1}
Fix a constant $L>0$, and a smooth positive function $b(x_n)$ on $[0,L]$. Then there exists a set of parameters $\mathfrak{P}_0^*\subset \R^2\times \R^+$ so that for any $(\bar\Phi_{en}, \msB_0, \bar p_{ex})\in \mathfrak{P}_0^*$, the system of \eqref{2-b2} and \eqref{2-b3d} in $\N$ with the boundary conditions \eqref{2-b8}--\eqref{bcvphi} has a unique one-dimensional smooth solution $(\vphi,\Phi)$ satisfying $\rho(\Phi, |\nabla\vphi|^2)\ge \underline{\rho}>0$ and \eqref{ellipticity} on $[0,L]$ for positive constants $\underline{\rho}$ and $\nu_1$ depending on $\bar{\Phi}_{en}, \msB_0, \bar{p}_{ex}$, $b$ and $L$.

\begin{proof}
Fix a constant $J_0>0$, and we first consider the initial value problem \eqref{ibp} on the interval $(0,L)$ with $b_0$ being replaced by a nonconstant positive function $b$. The initial data $(\rho_0,E_0)$ will be specified later. If $H'(\rho)\neq 0$ on $(0,L)$ for $H'(\rho)$ in \eqref{2-c8}, \eqref{ibp} is equivalent to
\begin{equation}
\label{ibp-nonconst-b}
\begin{cases}
\rho'=\frac{\rho^3E}{\rho^2p'(\rho)-J_0^2},\\
E'=\rho-b,
\end{cases}\quad
(\rho,E)(0)=(\rho_0,E_0),\quad u=\frac{J_0}{\rho}.
\end{equation}

\emph{Claim: For any given constant $J_0>0$, there exists a nonempty set $\mathfrak{P}_1^*(J_0)\subset \R^+\times \R$ so that for any $(\rho_0, E_0)\in \mathfrak{P}_1^*(J_0)$, the initial value problem \eqref{ibp-nonconst-b} has a unique $C^1$ solution $(\rho, E)(x_n)$ on the interval $[0,L]$ with satisfying $\rho>0$ and the inequality \eqref{2-a7} on $[0,L]$ for some positive constant $\nu_1$. Here, the choice of $\nu_1$ depends on $\rho_0, E_0, J_0$ and $L$.}

Once this claim is proved, one can repeat the argument in the proof of Proposition \ref{lemma-2-1} to complete the proof of Proposition \ref{proposition-app1}. So we only need to prove the claim.

Fix $J_0>0$ and $\rho_0\ge \sup_{[0,L]}b+\epsilon_0$ satisfying $p'(\rho_0)-(\frac{J_0}{\rho_0})^2\ge \nu_1>0$ for positive constants $\epsilon_0$ and $\nu_1$. We will show that if $E_0\ge \epsilon_1$ for some $\epsilon_1>0$, then the initial value problem \eqref{ibp-nonconst-b} has a unique $C^1$ solution $(\rho, E)$ on the interval $[0,L]$. Fix a constant $E_0\ge \epsilon_1$. Since $p'(\rho)$ is $C^1$, \eqref{ibp-nonconst-b} has a unique $C^1$ solution $(\rho, E)$ on $[0,\eps]$ for a constant $\eps>0$ so that the inequalities
\begin{equation*}
\rho\ge \sup_{[0,L]}b+\frac{\epsilon_0}{2},\quad
E\ge \frac{\epsilon_1}{2}\quad\tx{and}\quad p'(\rho_0)-\left(\frac{J_0}{\rho_0}\right)^2\ge \frac{\nu_1}{2}
\end{equation*}
 hold on $[0,\eps]$ where $\eps$ depends on $E_0$.
Using two equations in \eqref{ibp-nonconst-b}, one can directly check that
\begin{equation*}
\bigl(p'(\rho)-\frac{J_0^2}{\rho^2}\bigr)'=\frac{(\rho p''(\rho)+2(\frac{J_0}{\rho})^2)E}{p'(\rho)-(\frac{J_0}{\rho})^2}.
\end{equation*}
So if $\rho,E$ and $p'(\rho)-(\frac{J_0}{\rho})^2$ satisfy the three inequalities above on the $[0,\eps]$, then we get $\rho'(\eps)>0$, $E'(\eps)>0$ and $(p'(\rho)-(\frac{J_0}{\rho})^2)'(\eps)>0$ so \eqref{ibp-nonconst-b} is uniquely solvable beyond the interval $[0,\eps]$, moreover $(\rho, E)$ are increasing functions. Therefore, we can uniquely solve \eqref{ibp-nonconst-b} up to $x_n=L$, and the solution $(\rho, E)$ satisfy $\rho\ge \rho_0$, $E\ge E_0$ and $p'(\rho)-(\frac{J_0}{\rho})^2\ge \nu_1$ on $[0,L]$. For each $J_0>0$, we have shown that there exists a nonempty parameter set $\til{\mathfrak{P}}_1^*(J_0)$ so that, for any $(\rho_0,E_0)\in \til{\mathfrak{P}}_1^*(J_0)$, \eqref{ibp-nonconst-b} has a unique solution $(\rho, E)$ on $[0,L]$ with satisfying $\rho>0$ and \eqref{ellipticity} on $[0,L]$. Since $\til{\mathfrak{P}}_1^*(J_0)$ is nonempty, we can find the maximal set $\mathfrak{P}_1^*(J_0)$ such that, for any $(\rho_0,E_0)\in{\mathfrak{P}}_1^*(J_0)$, \eqref{ibp-nonconst-b} has a unique $C^1$ solution on the interval $[0,L]$ with satisfying $\rho>0$ and \eqref{ellipticity} on $[0,L]$. Set $\mathfrak{P}_1^*=\underset{J_0>0}{\cup} \mathfrak{P}_1^*(J_0)$. We repeat the argument in Proposition \ref{lemma-2-1} to get a set $\mathfrak{P}_0^*$ of parameters $(\bar{\Phi}_{en}, \msB_0, \bar{p}_{ex})$ by using a one-to-one and onto correspondence between $(\rho_0,E_0)\in\mathfrak{P}_1^*$ and $(\bar{\Phi}_{en}, \msB_0, \bar{p}_{ex})$. This completes the proof.
\end{proof}
\end{proposition}

\section{Perturbation of nozzle}
\label{section-perturbation}
For $n\ge 2$, let $\Lambda_0$ be an open, connected and bounded subset of $\R^{n-1}$ with $C^3$ boundary $\der\Lambda_0$, and set $\N_0$ to be an $n$-dimensional nozzle defined by
\begin{equation}
\label{app-a1}
\N_0:=\Lambda_0\times (0,L).
\end{equation}
For $0<\alp<1$, let $G:\Lambda_0\times (-L,2L)\to \R^{n-1}$ be a $C^{2,\alp}$ mapping such that $G(\cdot, x_n):\Lambda_0\to G(\Lambda_0\times\{x_n\})\subset \R^{n-1}$ is a diffeomorphism for every $x_n\in (-L, 2L)$ satisfying
\begin{equation}
\label{app-a2}
\|G(\cdot, x_n)-Id_{n-1}\|_{2,\alp,\Lambda_0}\le \sigma
\end{equation}
for a small constant $\sigma<0$ which will be specified later. For such a mapping $G$, we define a perturbation of $\N_0$ by
\begin{equation}
\label{app-a3}
\N_G:=\{(G(\rx),x_n)\in \R^n:\rx=(\rx',x_n), \rx'\in \Lambda_0, x_n\in(0,L)\}.
\end{equation}
$\N_G$ is a small perturbation of $\N_0$ in the sense that the wall boundary $G(\der\Lambda_0\times(0,L))$ of $\N_G$ is a perturbation of the wall boundary $\der\Lambda_0\times(0,L)$ of $\N_0$. Let $\mcl{T}$ denote the transformation $\mcl{T}(\rx',x_n)=(G(\rx',x_n),x_n)=:(\ry',y_n)$.
\begin{theorem}
\label{theorem-domain-pert}
Let $\N_0$ be defined by \eqref{app-a1}.
Fix $b_0>0$, $L>0$ and any $\alp\in(0,1)$, and let the parameter set $\mathfrak{P}_0$ be as in Proposition \ref{lemma-2-1}.
Given $(\bar\Phi_{en,0}, \msB_{0,0}, p_{ex,0})\in\mathfrak{P}_0$, let $(\vphi_0,\Phi_0)$ be the corresponding background solution.
We assume that $b, (\bar{\Phi}_{en},\bar\Phi_{ex},\pex)$ and $\msB_{0}$ are given small perturbations of $(\bar\Phi_{en,0},0, p_{ex,0})\in\mathfrak{P}_0$ and $\msB_{0,0}$, respectively, in the following sense
\begin{equation}
\label{data-app}
\begin{split}
&\|b-b_0\|_{\alp,\N_G}\le\sigma,\\
&\|\bar\Phi_{en}-\bar\Phi_{en,0}\|_{2,\alp,\mcl{T}(\Gam_0)}
+\|\bar\Phi_{ex}\|_{2,\alp,\mcl{T}(\Gam_L)}+\|\pex- p_{ex,0}\|_{\alp,\mcl{T}(\Gam_L)} \le \sigma,\\
&|\msB_0-\msB_{0,0}|\le \sigma.
\end{split}
\end{equation}
Also, we assume that $\bar{\Phi}_{en}$ and $\bar{\Phi}_{ex}$ satisfy
 \begin{equation}
\label{compatibility-app}
\der_{{\bf n}_w}\bar\Phi_{en}=0\;\;\tx{on}\;\;\mcl{T}(\ol{\Gam_0}\cap\ol{\Gam_w}),\quad
\der_{{\bf n}_w}\bar{\Phi}_{ex}=0\;\;\tx{on}\;\;\mcl{T}(\ol{\Gam_L}\cap\ol{\Gam_w})
\end{equation}
for inward unit normal vector ${\bf n}_w$ on $\mcl{T}(\der\Lambda_0\times (-L,2L))$, and that the perturbation $\mcl{T}$ satisfies \eqref{app-a2} and
\begin{equation}
\label{app-a4}
D_{\rx}\mcl{T}={I}_n\quad\tx{on}\quad \Gam_0\cup\Gam_L.
\end{equation}
Then, there exists a constant $\hat\sigma>0$ depending on the data and $\alp$ such that wherever $\sigma\in(0,\hat{\sigma}]$, the nonlinear boundary value problem
\begin{align}
\label{nbvp-1-app}
&\begin{cases}
div(\rho(\Phi,|\nabla\vphi|^2)\nabla\vphi)=0\\
\Delta\Phi=\rho(\Phi,|\nabla\vphi|^2)-b
\end{cases}\quad\tx{in}\quad\N_G\\
\label{nbvp-2-app}
&\vphi=0\;\;\tx{on}\;\;\Gam_0, \quad p(\rho(\Phi,|\nabla\vphi|^2))=p_{ex}\;\;\tx{on}\;\;\mcl{T}(\Gam_L)\\
\label{nbvp-3-app}
&\Phi=\begin{cases}
\msB_0+\bar{\Phi}_{en}&\tx{on}\;\;\mcl{T}(\Gam_0)\\
\msB_0+\bar{\Phi}_{ex}&\tx{on}\;\;\mcl{T}(\Gam_L)
\end{cases}\\
\label{nbvp-4-app}
&\der_{{\bf n}_w}\vphi=\der_{{\bf n}_w}\Phi=0\;\;\tx{on}\;\;\mcl{T}(\Gam_w).
\end{align}
has a unique solution $(\vphi,\Phi)\in[C^{2,\alp}_{(-1-\alp,{\mcl T}(\Gam_0)\cup {\mcl T}(\Gam_L))}(\N_G)]^2$ satisfying the following properties:
\begin{itemize}
\item[(a)] The equations in \eqref{nbvp-1-app} form a uniformly elliptic system in $\N_G$. Equivalently, the solution $(\vphi,\Phi)$ satisfies the inequality
    \begin{equation*}
    p'(\rho(\Phi,|\nabla\vphi|^2))-|\nabla\vphi|^2\ge \bar\nu>0\quad\tx{in}\quad\ol{\N_G}
    \end{equation*}
    for a positive constant $\bar{\nu}$, i.e., the flow governed by $(\vphi,\Phi)$ is subsonic;

\item[(b)] $(\vphi,\Phi)$ satisfy the estimate
   \begin{equation}
\label{2-c1}
\|\vphi-\vphi_0\|_{2,\alp,\N_G}^{(-1-\alp,\mcl{T}(\Gam_0)\cup\mcl{T}(\Gam_L))}+
\|\Phi-\Phi_0\|_{2,\alp,\N_G}^{(-1-\alp, \mcl{T}(\Gam_0)\cup\mcl{T}(\Gam_L))}\le C\sigma,
\end{equation}
for
$\sigma$ in \eqref{data-app}.
The constants $\bar\nu$ and $C$ depend only on  $b_0, L, \bar\Phi_{en,0}, \msB_{0,0}, p_{ex,0}, n,\Lambda$ and $\alp$.
\end{itemize}
\end{theorem}
Suppose that $(\vphi,\Phi)$ is a solutions to \eqref{nbvp-1-app}--\eqref{nbvp-4-app}, and let us set
\begin{equation*}
\phi(\rx):=\vphi\circ \mcl{T}(\rx),\quad W(\rx):=\Phi\circ \mcl{T}(\rx)\quad\tx{for}\quad \rx\in\N_0.
\end{equation*}
Then straightforward computation shows that $(\phi, W)$ satisfy
\begin{equation}
\label{app-a6}
\begin{split}
&\begin{cases}
\Div_{\rx} {\bf \mcl{A}}_1(\rx,W,\nabla\phi,J_{\mcl T})=0\\
\Div_{\rx}{\bf \mcl{A}}_2(\rx,\nabla W,J_{\mcl T})=\frac{1}{\det J_{\mcl T}}(\rho(W,|J_{\mcl T}\nabla\phi|^2)-b)
\end{cases}\quad\tx{in}\quad\N_0\\
&{\bf \mcl{A}}_1(\rx,W,\nabla\phi,J_{\mcl T})\cdot{\bf n}_w=
{\bf \mcl{A}}_2(\rx,\nabla W,J_{\mcl T})\cdot{\bf n}_w=0\quad\tx{on}\quad \Gam_w\\
&p(\rho(W,|J_{\mcl T}\nabla\phi|^2))=\pex\quad\tx{on}\quad \Gam_L
\end{split}
\end{equation}
for $J_{\mcl{T}}, {\bf{\mcl A}}_1$ and ${\bf{\mcl A}}_2$ defined by
\begin{equation*}
\begin{split}
&J_{\mcl{T}}(\rx):=\Bigl(\frac{\der(\mcl{T}^{-1})_j}{\der y_i}\Big|_{\ry=\mcl{T}(\rx)}\Bigr)_{i,j=1}^n,\\
&\mcl{A}_1(\rx,z,\q_1,M):=\bigl(\rho(z,|M\q_1^T|^2)\frac{M^TM}{\det M}\q_1^T\bigr)^T,\quad \mcl{A}_2(\rx,\q_2,M):=\bigl(\frac{M^TM\q_2^T}{\det M}\bigr)^T
\end{split}
\end{equation*}
for $(\rx,z,\q_1,\q_2,M)\in \N_0\times \R\times \R^n\times \R^n\times GL_n$. Here $\Div_{\rx}$ denotes $(\der_{x_1},\cdots, \der_{x_n})\cdot$, and $\q_1$ and $\q_2$ denote row vectors. And, ${\bf n}_w$ is inward unit normal vector on $\Gam_w$ in \eqref{app-a6}. We rewrite \eqref{app-a6} as
\begin{equation}
\label{app-a7}
\begin{split}
&\begin{cases}
\Div_{\rx} {\bf \mcl{A}}_1(\rx,W,\nabla\phi,I_n)=\Div_{\rx}(\underset{(=:{\bf H}_1)}{\underbrace{{\bf \mcl{A}}_1(\rx,W,\nabla\phi,I_n)-{\bf \mcl{A}}_1(\rx,W,\nabla\phi,J_{\mcl{T}})}})\\
\Div_{\rx}{\bf \mcl{A}}_2(\rx,\nabla W,I_n)=\frac{\rho(W,|J_{\mcl{T}}\nabla\phi|^2)-b}{\det J_{\mcl{T}}}+
\Div_{\rx}(\underset{(=:{\bf H}_2)}{\underbrace{{\bf \mcl{A}}_2(\rx,\nabla W,I_n)-{\bf \mcl{A}}_2(\rx,\nabla W,J_{\mcl{T}})}})\\
\end{cases}\tx{in}\quad\N_0\\
&\begin{cases}
{\bf \mcl{A}}_1(\rx,W,\nabla\phi,I_n)\cdot{\bf n}_w=({\bf \mcl{A}}_1(\rx,W,\nabla\phi,I_n)-{\bf \mcl{A}}_1(\rx,W,\nabla\phi,J_{\mcl{T}}))\cdot{\bf n}_w=:\mathfrak{g}_1\cdot{\bf n}_w\\
{\bf \mcl{A}}_2(\rx,\nabla W,J_{\mcl{T}})\cdot{\bf n}_w=({\bf \mcl{A}}_2(\rx,\nabla W,I_n)-{\bf \mcl{A}}_2(\rx,\nabla W,J_{\mcl{T}}))\cdot{\bf n}_w=:\mathfrak{g}_2\cdot{\bf n}_w
\end{cases}\quad\tx{on}\quad \Gam_w\\
&p(\rho(W,|\nabla\phi|^2))=\pex+
\Bigl( p(\rho(W,|\nabla\phi|^2))-p(\rho(W,|J_{\mcl{T}}\nabla\phi|^2))
\Bigr)
=:\pex+\mathfrak{g}_3\quad\tx{on}\quad \Gam_L
\end{split}
\end{equation}
so that the left-hand sides of equations in \eqref{app-a7} are same as the left-hand sides of \eqref{nbvp-1}, and the boundary conditions in \eqref{app-a7} can be rewritten as conormal boundary conditions. Also, if $(\phi, W)$ are sufficiently close to $(\vphi_0,\Phi_0)$ in $C^{2,\alp}_{(-1-\alp,\corners)}(\N_0)$, then, by \eqref{app-a2}, there exists a constant $C$ depending only on the data and $\alp$ so that ${\bf H}_1, {\bf H}_2, \mathfrak{g}_1, \mathfrak{g}_2$ and $\mathfrak{g}_3$ satisfy
\begin{equation}
\label{app-b1}
\begin{split}
&\sum_{l=1}^2\|{\bf H}_l\|_{1,\alp,\N_0}^{(-\alp,\corners)}+\sum_{k=1}^3\|\mathfrak{g}_k\|_{1,\alp,\N_0}^{(-\alp,\corners )}
\le C\|I_n-J_{\mcl{T}}\|_{1,\alp,\N_0}\le C\sigma.
\end{split}
\end{equation}
The condition \eqref{app-a4} implies that $J_{\mcl{T}}=I_n$ on $\Gam_0\cup\Gam_L$ so we have
${\bf H}_1={\bf 0}={\bf H}_2\quad\tx{on}\;\;\Gam_0\cup \Gam_L$, then ${\bf H}_1$ and ${\bf H}_2$ satisfy the compatibility condition \eqref{F-condition} on $\Gam_0\cup\Gam_L$.
Now we can repeat the argument in $\S$\ref{section-3} and $\S$\ref{section4} to prove Theorem \ref{theorem-domain-pert}.

\bigskip

{\bf Acknowledgement.}
The research of Myoungjean Bae was supported in part by Priority Research Centers Program under Grant 2012047640  and by the Basic Science Research Program under Grant 2012R1A1A1001919 through the National Research Foundation of Korea(NRF) funded by the Ministry of Education, Science and Technology.
The research of Ben Duan was supported in part by Priority Research Centers Program under Grant 2012047640. The research of Chunjing Xie was supported in part by NSFC 11241001, 11201297 and Shanghai Pujiang Talent Program. Part of the work was done when Chunjing Xie was visiting POSTECH, and when Myoungjean Bae was visiting Shanghai Jiao Tong University and National Center for Theoretical Sciences(Taiwan), and when Ben Duan was visiting Shanghai Jiao Tong University. They thank these institutes for the warm hospitality and support during these visits.


\bigskip


\begin{thebibliography}{99}




\bibitem{MarkPhase}
U.~M. Ascher, P. A. Markowich, P. Pietra, and C.
Schmeiser, {\it A phase plane analysis of transonic solutions for
the hydrodynamic semiconductor model}, Math. Models Methods Appl.
Sci. {\bf 1} (1991), no. 3, 347--376.


%\bibitem{Bae-F}
%M. Bae and M. Feldman, {\it Transonic shocks in multidimensional divergent nozzles},
%{\em Arch. Ration. Mech. Anal.}  (2011), no. 3, 777--840.


\bibitem{Shu}
D. P. Chen, R. S. Eisenberg, J. W. Jerome, C. W. Shu, {\it A
hydrodynamic model of temperature change in open ionic channels},
Biophys J. {\bf 69} (1995), 2304-2322.


%\bibitem{CF1}
%G.-Q. Chen and M. Feldman,  {\it Multidimensional transonic
%shocks and free boundary problems for nonlinear equations of mixed
%type}, J. Amer. Math. Soc. {\bf 16} (2003), no. 3, 461--494.

\bibitem{Ch-F3}
 G.-Q. Chen and M.  Feldman,
 {\it Existence and stability of multidimensional transonic flows through an infinite nozzle
of arbitrary cross-sections}, {Arch. Rational Mech. Anal.} {\bf 184} (2007), no 2, 185--242.

\bibitem{ChenWang}
G.-Q. Chen and  D.-H. Wang, {\it Convergence of shock capturing schemes
for the compressible {E}uler-{P}oisson equations}, Comm. Math. Phys.
{\bf 179} (1996),  no. 2, 333--364.


\bibitem{WangChen}
G.-Q. Chen and D.-H. Wang, {\it Formation of singularities in
compressible {E}uler-{P}oisson fluids with heat diffusion and
damping relaxation}, J. Differential Equations  {\bf 144} (1998), no.
1, 44--65.


\bibitem{Schen-Yu1}
S.-X. Chen and H.-R. Yuan,
{\it Transonic shocks in compressible flow passing a duct for three dimensional Euler system},\newblock{ Arch. Ration. Mech. Anal.} {\bf 187}(2008), no 3, 523--556

\bibitem{DeMark1d}
P. Degond and  P. A. Markowich, {\it On a one-dimensional
steady-state hydrodynamic model for semiconductors}, Appl. Math.
Lett. {\bf 3} (1990),  no. 3, 25--29.


\bibitem{DeMark3d}
P. Degond and  P. A. Markowich, {\it A steady state potential flow
model for semiconductors}, Ann. Mat. Pura Appl. (4)  {\bf 165}
(1993), 87--98.




\bibitem{Gamba1d}
I. M. Gamba, {\it Stationary transonic solutions of
a one-dimensional hydrodynamic model for semiconductors}, Comm.
Partial Differential Equations {\bf 17} (1992), no. 3-4, 553--577.


\bibitem{GambaMorawetz}
I. M. Gamba and C. S. Morawetz, {\it A viscous
approximation for a {$2$}-{D} steady semiconductor or transonic gas
dynamic flow: existence theorem for potential flow}, Comm. Pure
Appl. Math. {\bf 49} (1996),  no. 10, 999--1049.


\bibitem{GilbargTrudinger}
D. Gilbarg, and N. Trudinger,
{\it  Elliptic Partial Differential Equations of Second Order.}  2nd Ed. Springer-Verlag: Berlin.

\bibitem{Guo}
Y. Guo and W. Strauss, {\it Stability of semiconductor states with insulating and contact boundary conditions},
 Arch. Ration. Mech. Anal. {\bf 179} (2006), no. 1, 1--30.


\bibitem{Ha-L}
Q. Han and F. Lin,
{\it Elliptic partial differential equations.}  Courant Institute of Math. Sci., NYU.

\bibitem{Huang}
F. Huang, R. Pan and H. Yu,  {\it Large time behavior of Euler-Poisson system for semiconductor},
Sci. China Ser. A {\bf 51} (2008), no. 5, 965--972.



\bibitem{LiMark}
H. Li, P. A. Markowich and M. Mei,  {\it Asymptotic behavior of subsonic entropy solutions of the isentropic Euler-Poisson equations},  Quart. Appl. Math.  {\bf 60} (2002), no. 4, 773--796.

\bibitem{LXY1}
J. Li, Z.-P. Xin and H.-C. Yin, {\it On transonic shocks in a
nozzle with variable end pressures},  Comm. Math. Phys. {\bf 291} (2009), no. 1, 111--150.

%\bibitem{LXY2}
%J. Li, Z.-P. Xin and H.-C. Yin, {\it A free boundary value
%problem for the Euler system and 2-D transonic shock in a large
%variable nozzle},   Math. Res. Lett.  {\bf 16} (2009), no. 5, 777--796.


%\bibitem{lieberman2}
%G. Lieberman, {\it Mixed boundary value problems for elliptic and parabolic differential equations of second order},
%{ J. Math. Anal. Appl.} {\bf 113}, (1986), 422--440

\bibitem{lieberman1}
G. Lieberman, {\it H\"{o}lder continuity of the gradient of solutions of uniformly parabolic equations with conormal boundary conditions},
{Ann. Mat. Pura Appl.} {\bf(4), 148} (1987), 77--99


\bibitem{LuoNX}
T. Luo, R. Natalini and Z.-P. Xin, {\it Large time behavior of the solutions to a hydrodynamic model for semiconductors},  SIAM J. Appl. Math. \textbf{59} (1999), no. 3, 810--830.

\bibitem{LRXX}
T. Luo, J. Rauch, C. Xie and Z.-P. Xin, {\it Stability of transonic shock solutions for one-dimensional Euler-Poisson equations}, Arch. Ration. Mech. Anal. {\bf 202} (2011), no. 3, 787--827

\bibitem{LuoXin}
T. Luo and Z.-P. Xin, {\it Transonic shock solutions for a
system of Euler-Poisson equations}, Comm. Math. Sci. {\bf 10} (2012), no. 2, 419--462.



\bibitem{MarkZAMP}
P. A. Markowich, {\it On steady state {E}uler-{P}oisson models
for semiconductors}, Z. Angew. Math. Phys. {\bf 42} (1991), no. 3,
389--407.


\bibitem{MarkRSbook}
P. A. Markowich, C. A. Ringhofer, and C. Schmeiser, {\it
Semiconductor equations}, Springer-Verlag, Vienna, 1990.


\bibitem{Peng}
Y.-J. Peng and I. Violet,  {\it Example of supersonic
solutions to a steady state {E}uler-{P}oisson system}, Appl. Math.
Lett. {\bf 19} (2006),  no. 12, 1335--1340.


\bibitem{RosiniStability}
M. D. Rosini, {\it Stability of transonic strong shock
waves for the one-dimensional hydrodynamic model for
semiconductors}, J. Differential Equations {\bf 199} (2004),  no. 2,
326--351.


\bibitem{RosiniPhase}
M. D. Rosini, {\it A phase analysis of transonic solutions
for the hydrodynamic semiconductor model}, Quart. Appl. Math. {\bf
63} (2005),  no. 2, 251--268.


\bibitem{XinYinCPAM}
Z.-P. Xin and H.-C. Yin, {\it Transonic shock in a nozzle. I.
Two-dimensional case}, Comm. Pure Appl. Math.  {\bf 58} (2005), no.
8, 999--1050.

%\bibitem{XinYinJDE}
%Z.-P.Xin and H.-C. Yin, {\it The transonic shock in a nozzle, 2-D and 3-D complete Euler systems},
%  J. Differential Equations {\bf 245} (2008), no. 4, 1014--1085.

\bibitem{ZhangBo}
B. Zhang,  {\it Convergence of the {G}odunov scheme for a simplified
one-dimensional hydrodynamic model for semiconductor devices}, Comm.
Math. Phys. {\bf 157} (1993), no. 1, 1--22.

\end{thebibliography}
\end{document}